\documentclass{amsart}

\usepackage{amsmath}
\usepackage{amssymb}  
\usepackage{amsthm}   
\usepackage{natbib}

\numberwithin{equation}{section}

 \DeclareMathOperator{\tor}{Tor}
\DeclareMathOperator{\lcm}{lcm} 

\DeclareMathOperator{\Ann}{Ann}
\DeclareMathOperator{\supp}{supp}\DeclareMathOperator{\Char}{char}

\DeclareMathOperator{\Span}{Span}

\theoremstyle{plain}
\newtheorem{thm}{Theorem}[section]
\newtheorem{lem}[thm]{Lemma}
\newtheorem{prop}[thm]{Proposition}
\newtheorem{cor}[thm]{Corollary}
\theoremstyle{definition}
\newtheorem{defn}[thm]{Definition}

\newtheorem{exmp}[thm]{Example}
\newtheorem*{rem}{Remark}
\theoremstyle{remark}
\newtheorem{case}{Case}
\newtheorem{subcase}{Case}

\def\A{\mathbb{A}}

\def\N{\mathbb{N}}

\begin{document}

    \begin{center}
        \LARGE Dynatomic cycles for morphisms of projective varieties\\
        \bigskip
        \normalsize By \it Benjamin A. Hutz
    \end{center}

    \paragraph{Abstract.}
        We prove the effectivity of the zero-cycles of formal periodic points, dynatomic cycles, for morphisms of projective varieties.  We then analyze the degrees of the dynatomic cycles and multiplicities of formal periodic points and apply these results to the existence of periodic points with arbitrarily large primitive periods.


     \section{Introduction}
        Let $K$ be a field and $X/K$ a projective variety.  Let $\phi: X/K \to X/K$ be a morphism
        defined over $K$.  We can iterate the morphism $\phi$ and study the properties of the
        periodic points of the resulting dynamical system.  In this paper, we consider
        $K$ an algebraically closed field and study the zero cycle of
        formal $n$-periodic points, the \emph{$n$-th dynatomic cycle}, and show that it is effective for all positive integers $n$ (Theorem \ref{thm15}), resolving
        a conjecture of Morton and Silverman in the affirmative \cite[Conjecture 1.1]{Silverman6}.  We further show
        that the periodic points of formal period $n \neq 0$ in $K$ with multiplicity one
        have primitive period $n$ (Theorem \ref{thm17}).  We relate the degrees of the dynatomic cycles
        to periodic Lefschetz numbers and use information about the degrees to investigate the
        existence of primitive periodic points.  In particular, we show that the
        dynamical systems constructed on Wehler K3 surfaces and the dynamical systems arising from morphisms of
        projective space have periodic points with arbitrarily large primitive
        periods (Theorem \ref{thm19} and Theorem \ref{thm20}).  Much of this work is
        from the author's doctoral thesis \cite[Chapter 3]{Hutz}.

        We now describe our results in more detail.
        We denote $\phi^n$ as the $n$-th iterate of the morphism $\phi$.
        If $\phi^n(P) = P$ for some $P \in X(K)$ and $n \in \mathbb{N}$, then $P$ is
        called a \emph{periodic point of period $n$} for $\phi$.  If $n$ is the smallest such
        period, then $n$ is called the \emph{primitive} period of $P$.  Consider the cycles in $X \times X$: The graph of $\phi^n$ defined as
        $\Gamma_n = \sum_{x \in X} (x,\phi^{n}(x))$ and the diagonal $\Delta = \sum_{x \in X}(x,x)$.
        \begin{defn}
            For $n \geq 1$, we say that $\phi^{n}$ is \emph{non-degenerate} if $\Delta$ and $\Gamma_n$
            intersect properly.
        \end{defn}
        \begin{rem}
            If $\phi^n$ is non-degenerate, then $\phi^d$ is non-degenerate for all $d \mid n$.
        \end{rem}
        Assume that $\phi^n$ is non-degenerate, and let $P \in X(K)$.
        Define $a_P(n)$ to be the intersection multiplicity of $\Gamma_n$ and $\Delta$ at $(P,P)$ and \begin{equation*}
            \Phi_n(\phi) = \sum_{P \in X} i(\Gamma_n,\Delta;P) (P) = \sum_{P \in X} a_P(n)(P).
        \end{equation*}
        Notice that this intersection contains all of the periodic points of period $n$
        for $\phi$.  We want to examine only the primitive $n$-periodic points.  Define
        \begin{equation*}
            a_P^{\ast}(n) = \sum_{d \mid n} \mu\left(\frac{n}{d}\right) a_P(n)
        \end{equation*}
        and
        \begin{equation*}
            \Phi_n^{\ast}(\phi) = \sum_{d \mid n} \mu\left(\frac{n}{d}\right)
            \Phi_d(\phi) = \sum_{P \in X} a_{P}^{\ast}(n) (P),
        \end{equation*}
        where $\mu$ is the M\"{o}bius function.
        \begin{defn}
            We call $\Phi^{\ast}_n(\phi)$ the \emph{$n$-th dynatomic cycle} and $a_P^{\ast}(n)$ the
            \emph{multiplicity} of $P$ in $\Phi^{\ast}_n(\phi)$.  If
            $a_P^{\ast}(n) >0$, then we call $P$ a \emph{formal}
            periodic point of \emph{formal} period $n$.
        \end{defn}
        All periodic points of primitive period $n$ are points of formal period $n$
        (Proposition \ref{prop16}(\ref{prop16b})), but there may be periodic points with formal period $n$ and primitive period strictly less than $n$ \cite[Theorem 2.4]{Morton}.

        For $\phi:\A^1 \to \A^1$, a single variable polynomial map, Morton showed that $\Phi^{\ast}_n(\phi)$ is effective; and in the case where $n \neq 0$ in $K$ and it has points with multiplicity greater than one, its points of nonzero multiplicity are exactly the points of primitive period $n$ \cite[Theorem 2.4 and Theorem 2.5]{Morton}.  Morton and Silverman went on to show the effectivity of $\Phi^{\ast}_n(\phi)$ for a non-degenerate morphism of a nonsingular projective curve and for a non-degenerate \emph{auto}morphism of projective space \cite[Theorem 2.1 and Theorem 3.1]{Silverman6}.  They also conjectured effectivity for non-degenerate morphisms of nonsingular projective varieties \cite[Conjecture 1.1]{Silverman6}.

        In Section \ref{sect1} we prove that $\Phi^{\ast}_n(\phi)$ is effective for non-degenerate
        morphisms of non-singular, irreducible, projective varieties
        and describe the possible values of $n$ for which a periodic point $P$ of $\phi$ has
        non-zero multiplicity in $\Phi^{\ast}_n(\phi)$ (Theorem \ref{thm15}).

        As in the one-dimensional case, the proof is carried out
        by carefully examining when the multiplicity of a fixed point $P$ in $\Phi_n(\phi)$ is greater
        than the multiplicity of $P$ in $\Phi_1(\phi)$.  However, several new
        ideas and a lot of additional work are needed in the higher dimensional case.  Some of the
        difficulties encountered are taking into account the higher $\tor$ modules
        in the intersection theory, which turn out to all be identically 0 (Theorem \ref{thm10}),
        using the theory of standard bases to obtain information about the multiplicity of a
        point in $\Phi_n(\phi)$ (Proposition \ref{prop10}), and iterating local power
        series representations of the morphism.

        From this detailed analysis of the multiplicities, in Section \ref{formal_periodic} we show that
        periodic points of formal period $n$ with multiplicity one and $n\neq 0$ in $K$ have
        primitive period $n$.  In other words, $a_P^{\ast}(n) =1$ for $n \neq 0$ in $K$
        implies that $P$ is a periodic point of primitive period $n$.  This generalizes
        \cite[Theorem 2.5]{Morton} to morphisms of projective varieties.

        In Section \ref{basic_properties} we state some basic properties of $\Phi_n(\phi)$ and
        $\Phi^{\ast}_n(\phi)$.
        In Section \ref{lefschetz} we note the similarity to periodic Lefschetz numbers, and in Section \ref{applications} we state results similar to those of \cite{Fagella,Llibre,Shub} on the existence of primitive periodic points.  In particular, if $P$ is a periodic point, then the sequence $a_P(n)$ for $n \neq 0$ in $K$ is bounded (Theorem \ref{thm22}), and if $\deg(\Phi_n)$ is unbounded for $n \neq 0$ in $K$, then there are periodic points with arbitrarily large primitive periods and infinitely many periodic points (Corollary \ref{cor10}).
        In Section \ref{wehler} these results are applied to dynamical systems on Wehler K3 surfaces studied in
        \cite{Silverman4, Silverman5} and in Section \ref{pn} to dynamical systems arising from morphisms of projective space.

        The cycles $\Phi_n(\phi)$ and $\Phi^{\ast}_n(\phi)$ occur with great frequency in the
        literature, under a variety of notations and with a number of results
        stemming from the fact they are effective, see for example
        \cite{Morton2, Morton4, Morton5, Morton, Silverman6, Silverman7, Morton3, Vivaldi}.
        In particular, \cite{Morton,Vivaldi} contain Galois theoretic results
        in the single-variable polynomial case where $\Phi^{\ast}_n(\phi)$ has no points of multiplicity greater than one; many of the arguments of these two articles carry through to the higher dimensional case given that $\Phi^{\ast}_n(\phi)$ is effective (see \cite[Chapter 4]{Hutz}).

        The author would like to thank his advisor Joseph Silverman for his many
        insightful suggestions and ideas and also Dan Ambramovich and Michael Rosen
        for their suggestions.  The author also thanks Jonathan Wise for showing him the proof of Proposition \ref{prop20}(\ref{prop20a}) and Michelle Manes for her contribution to Proposition \ref{prop22}.

    \section{Effectivity of $\Phi^{\ast}_n(\phi)$} \label{sect1}
        Let $K$ be an algebraically closed field and let $X$ be a non-singular, irreducible, projective variety of
        dimension $b$ defined over $K$.  Let $\phi:X \to X$ be a morphism defined over $K$ such that $\phi^n$ is non-degenerate.  Define the cycles in $X \times X$: The graph of $\phi^n$ defined as
        $\Gamma_n = \sum_{x \in X} (x,\phi^{n}(x))$ and the diagonal $\Delta = \sum_{x \in X}(x,x)$. Let $R_P$ be the local ring of $X \times X$ at $(P,P)$ and let $I_{\Delta},
        I_{\Gamma_n} \subset R_P$ be the ideals of $\Delta$ and $\Gamma_n$, respectively.
        The following steps outline the proof of the effectivity of $\Phi^{\ast}_n(\phi)$.
        \begin{enumerate}
            \item Define the intersection multiplicity and show that $\Phi^{\ast}_n(\phi)$ is a zero-cycle.
            \item Show that the naive intersection theory is, in fact,
                correct (Theorem \ref{thm10}). Specifically, show that
                \begin{equation*}
                    \tor_i(R_P/I_{\Delta},R_P/I_{\Gamma_n})=0 \quad \text{for all } i >0.
                \end{equation*}
            \item Determine conditions on $n$ for when $a_P(n) > a_P(1)$ (Proposition \ref{prop10}).
            \item Show that $a_{P}^{\ast}(n) \geq 0$ for all $P$ and $n$ (Theorem \ref{thm15}).
        \end{enumerate}
        In what follows, the concept of dimension will be used in several different
        contexts.  We will denote
        \begin{itemize}
            \item $\dim R$ for the Krull dimension of a ring $R$,
            \item $\dim M$ for the Krull dimension of $R/\Ann(M)$ where $\Ann(M)$ is the
                annihilator of the $R$-module $M$, and
            \item $\dim_K V$ for the dimension of the finite dimensional $K$-vector space $V$.
        \end{itemize}

    \subsection{Intersection multiplicity}
        Using Serre's definition of intersection multiplicity and working over the completion
        $\widehat{R_P}$ of $R_P$, we have
        \begin{equation*}
            i(\Delta,\Gamma_n;P) = \sum_{i=0}^{b -1} (-1)^{i}
            \dim_K(\tor_{i}(\widehat{R_P}/I_{\Delta},\widehat{R_P}/I_{\Gamma_n})).
        \end{equation*}

        Since $\phi^n$ is non-degenerate, the cycles $\Delta$ and $\Gamma_n$ intersect
        properly.  We also know $X \times X$ has dimension $2b$, $\Delta$ has dimension $b$, and
        $\Gamma_n$ has dimension $b$.  Consequently, $\Phi_n(\phi)$ is a zero-cycle with a finite number
        of points with non-zero multiplicity.  Therefore,
        $\Phi^{\ast}_n(\phi)$ is also a zero-cycle.

        In local coordinates, we have
        \begin{equation*}
            \widehat{R_P} \cong K[[x_1,\ldots,x_b,y_1,\ldots,y_b]].
        \end{equation*}
        \begin{defn}
            Let $\phi(\textbf{x}) = [\phi_1(\textbf{x}),\ldots,\phi_b(\textbf{x})]$, where
            $\textbf{x} = (x_1,\ldots,x_b)$.  Then denote
            \begin{equation*}
                \phi^n(\textbf{x}) = [\phi_1^{(n)}(\textbf{x}),\ldots,\phi_b^{(n)}(\textbf{x})]
            \end{equation*}
            as the coordinates of the $n$-th iterate of $\phi$.
        \end{defn}
        Then we have
        \begin{equation*}
            I_{\Delta} = (x_1-y_1,\ldots, x_b-y_b) \quad \text{and} \quad
            I_{\Gamma_n} = (\phi_1^{(n)}(\textbf{x})-y_1, \ldots, \phi_b^{(n)}(\textbf{x})-y_b).
        \end{equation*}

        We will use the non-degeneracy of $\phi^n$ and the following theorem to show that $\tor_i(R_P/I_{\Delta},R_P/I_{\Gamma_n}) =0$ for all $i>0$.
        \begin{thm}\textup{(\cite[Corollary to Theorem V.B.4]{Serre})} \label{thm9}
            Let $(R,\mathfrak{m})$ be a regular local ring of dimension $b$, and let $M$ and $N$ be two
            non-zero finitely generated $R$-modules such that $M \otimes N$ is of finite
            length.  Then $\tor_i(M,N) = 0$ for all $i>0$ if and only if $M$ and $N$ are
            Cohen-Macaulay modules and $\dim M + \dim N = b$.
        \end{thm}

        \begin{thm} \label{thm10}
            Let $X$ be a non-singular, irreducible, projective variety defined over a
            field $K$ and $\phi:X \to X$ a morphism defined over $K$ such that $\phi^n$
            is non-degenerate.  Let $P \in X(K)$ and $R_P$ the local ring of $X \times X$
            at $(P,P)$.  Let $\Delta,\Gamma_n \subset X \times X$ be the diagonal and the graph
            of $\phi^n$, respectively, and let $I_{\Delta}, I_{\Gamma_n} \subset R_P$ be their
            ideals. Then, $\tor_i(R_P/I_{\Delta},R_P/I_{\Gamma_n}) =0$ for all $i>0$.
        \end{thm}

        \begin{proof}
            Let $b=\dim{X}$, then we have $\dim{X \times X} = 2b$ and $\dim{\Delta}=\dim{\Gamma_n}=b$.  The ideals $I_{\Delta}$ and $I_{\Gamma_n}$ are each generated by $b$ elements and $\Delta$ and $\Gamma_n$ intersect properly.  Therefore,
            \begin{equation*}
                \dim_K(R_P/(I_{\Delta} + I_{\Gamma_n})) = \text{length}(R_P/I_{\Delta} \otimes R_P/I_{\Gamma_n}) < \infty.
            \end{equation*}

            By \cite[Proposition III.B.6]{Serre} the union of
            the generators of $I_{\Delta}$ and the generators of $I_{\Gamma_n}$ are a system
            of parameters for $R_P$. Because $R_P$ is Cohen-Macaulay by \cite[Corollary 3 to Theorem IV.D.9]{Serre} we can apply \cite[Corollary to Theorem IV.B.2]{Serre} to
            $I_{\Delta}$ and its generators to conclude that $R_P/I_{\Delta}$ is Cohen-Macaulay of dimension $b$ and, similarly with $I_{\Gamma_n}$, to conclude that $R_P/I_{\Gamma_n}$ is Cohen-Macaulay of
            dimension $b$.

            We have fulfilled the hypotheses of Theorem \ref{thm9}; consequently, we have that
            \begin{equation*}
                \tor_i(R_P/I_{\Delta},R_P/I_{\Gamma_n}) =0 \quad \text{for all } i > 0.
            \end{equation*}
        \end{proof}

    \subsection{$\tor_0$ module}
        If $P$ is not a periodic point, then $a_P(n) =0 $ for all $n$, so we will assume that $P$ is a periodic point.  If $a_P(1) = 0$, then $P$ has some primitive period $m>1$.  If $m \nmid n$ then $a_P(n)=0$, so we may replace $\phi$ by $\phi^m$ and assume that $P$ is a fixed point for $\phi$ and, hence, $a_P(1) > 0$.   For $P$, a fixed point of $\phi$, we can iterate a local representation of $\phi$ as a family of power series.

        From Theorem \ref{thm10} we know the naive intersection index
        \begin{align*}
            a_P(&n) = i(\Delta,\Gamma_n;P) \\
            &= \dim_K(\tor_0(R_P/I_{\Delta},R_P/I_{\Gamma_n})) \\
            &= \dim_K(R_P/(I_{\Delta}+I_{\Gamma_n}))
        \end{align*}
        is, in fact, correct in our situation.  To prove the effectivity of $\Phi^{\ast}_n(\phi)$,
        we will use conditions on $n$ for $a_P(n)$ to be greater than $a_P(1)$.  To determine these conditions, we will consider local power series representations of $\phi$ and the theory of standard bases.  For information on standard bases, see \cite[Chapter 4]{CLO2}.  Below, we recall the needed terminology.

        \begin{defn}
            Recall that for a ring of formal power series $K[[X_1,\ldots,X_h]]$, we can write an element $f \in K[[X_1,\ldots,X_h]]$ as
            \begin{equation*}
                f = \sum_{v \in \mathbb{N}^{h}} f_vX^{v}.
            \end{equation*}
            The monomial \emph{support of $f$} is defined as
            \begin{equation*}
                \supp(f) = \{f_vX^v \mid f_v \neq 0\}.
            \end{equation*}
            If $f \neq 0$, then $\supp(f)$ has a least element under any admissible monomial ordering.  We call this least element the \emph{leading monomial of $f$}, denoted by $LM(f)$.  We denote $v(f)$ the exponent of the leading monomial.  Then
            \begin{equation*}
                LM(f) = f_{v(f)}X^{v(f)}
            \end{equation*}
            and we call $X^{v(f)}$ the \emph{leading term of $f$} and denote it by
            $LT(f)$.

            Let $I$ be an ideal in $K[[X_1,\ldots,X_h]]$.  We define the \emph{leading term ideal of $I$}
            as
            \begin{equation*}
                LT(I)= \text{ the polynomial ideal generated by } \{X^{v} \mid \exists f \in I \text{ with }
                LT(f) = X^{v}\}.
            \end{equation*}
        \end{defn}

        \begin{defn} \label{defn4}
            A non-zero element $f \in K[[X_1,\ldots,X_h]]$ is called \emph{self-reduced}
            with respect to an admissible monomial ordering if
            \begin{equation*}
                LT(f) \nmid F \text{ for all } F \in \supp(f)-LT(f).
            \end{equation*}
        \end{defn}

        Finally we recall three facts that we will need (see \cite[Chapter 4.4]{CLO2}).
        \begin{thm} \label{thm13}
            The following are equivalent.
            \begin{enumerate}
                \item There exists a standard basis for $I$.
                \item Every $f \in K[[X_1,\ldots,X_h]]$ has a unique standard
                    remainder modulo $I$.
                \item Every $f \in K[[X_1,\ldots,X_h]]$ has a standard remainder
                modulo $I$.
        \end{enumerate}
        \end{thm}

        \begin{thm} \label{thm18}
            Every ideal $I \subset K[[X_1,\ldots,X_h]]$ has a universal standard basis.
        \end{thm}

        \begin{thm} \label{thm1}
            Let $I \subset K[[X_1,\ldots,X_h]]$ be an ideal with $\dim K[[X_1,\ldots,X_h]]/I=0$.  Then $K[[X_1,\ldots,X_h]]/I$ is isomorphic as a $K$-vector space to $\Span(X^{v} \mid X^{v} \not\in LT(I))$.
        \end{thm}
        For the most part, we will not be concerned with the particular admissible ordering that is used, so it what follows we fix an admissible monomial ordering.  When necessary, we will specify a particular ordering.

        \begin{rem}
            For notation convenience, define $I_n = I_{\Delta} + I_{\Gamma_n}$.
        \end{rem}
        \begin{cor} \label{cor9}
            Consider the ideal $I_n=(I_{\Delta}+I_{\Gamma_n}) \subset \widehat{R_P}$.  Then
            \begin{equation*}
                a_P(n)=\dim_K(\widehat{R_P}/I_n)= \dim_K(\Span(X^{v} \mid X^{v} \not\in
                LT(I_n))).
            \end{equation*}
        \end{cor}

        \begin{proof}
            Apply Theorem \ref{thm1} to $\widehat{R_P}$ and $I_n$.
        \end{proof}

        \begin{lem} \label{lem10}
            Assume $\phi^n$ is non-degenerate.  Then $a_P(n)\geq a_P(1)$ for all $n \in \N$.
        \end{lem}

        \begin{proof}
            It is clear that
            \begin{equation*}
                \Gamma_1 \cap \Delta \subseteq \Gamma_n \cap \Delta
            \end{equation*}
            and we have a local representation of $\phi = (\phi_1,\ldots,\phi_n)$ at the fixed point $P$.  Iterating this representation involves taking combination of the $\phi_i$ and hence are all elements of the original ideal $I_{\Gamma_1}$.  Hence, we have
            \begin{equation*}
                I_{\Gamma_n} + I_{\Delta}=I_n \subseteq I_1 = I_{\Gamma_1} + I_{\Delta}.
            \end{equation*}
            Therefore,
            \begin{equation*}
                LT(I_n) \subseteq LT(I_1)
            \end{equation*}
            which implies $a_P(n) \geq a_P(1)$.
        \end{proof}

        For a fixed point $P$, if $a_P(n) > a_P(1)$, then some monomial of a power series representation of $\phi$ near $P$ does not appear in the $n$-th iterate of that representation.  So we need to examine when monomials may have coefficient $0$ after iteration.  We next show that we may reduce to the case where the generators of the ideal are self-reduced.

        \begin{lem} \label{lem8}
            Let $I \subset K[[X_1,\ldots,X_h]]$ be an ideal generated by
            $\{f_1,\ldots,f_m\}$ with $\dim K[[X_1,\ldots,X_h]]/I = 0$.  Let $u_i \in K[[X_1,\ldots,X_h]]$ be a unit such that $u_if_i$ is self-reduced for each $1 \leq i \leq h$ and define $uI = (u_1f_1,\ldots,u_mf_m)$. Then
            \begin{equation*}
                \dim_K(\Span(X^{v} \mid X^{v} \not\in LT(I))) =
                \dim_K(\Span(X^{v} \mid X^{v} \not\in LT(uI))).
            \end{equation*}
        \end{lem}

        \begin{rem}
            By \cite[Corollary 2.2]{Becker} applied to
            \begin{equation*}
                LT(I) = LT((f_1,\ldots,f_m)),
            \end{equation*}
            we know each there exist units $u_i \in K[[X_1,\ldots,X_h]]$ such that each $u_if_i$ is self-reduced.
        \end{rem}

        \begin{proof}
             Since each $u_i$ is a unit, we have $v(LT(u_i))=0$ and $LT(u_if_i) = LT(f_i)$ (and similarly for any combinations of the $f_i$).  Hence we have
            \begin{equation*}
                LT(I)=LT((f_1,\ldots,f_m)) =LT((u_1f_1,\ldots,u_mf_m)).
            \end{equation*}
        \end{proof}

        We now show that we can also exclude from consideration those monomials that are products of other monomials in the monomial support of $\phi$ under iteration.
        \begin{exmp} \label{exmp1}
            We have
            \begin{align*}
                \phi_1(x,y,z) &= x^4 + \fbox{$x^2z^2$} + xy\\
                \phi_2(x,y,z) &= y^4 + xz^2\\
                \phi_3(x,y,z) &= z^4.
            \end{align*}
            We have that $xy$ creates an additional $x^2z^2$ term under iteration.
        \end{exmp}
        \begin{lem} \label{lem1}
            Assume that $H \in \supp(\phi_i)$ is a monomial which is a product of other monomials under iteration.  Then $LT(\phi_1, \ldots,\phi_d) = LT(\phi_1,\ldots,\phi_i-H,\ldots,\phi_d)$.
        \end{lem}
        \begin{proof}
            Assume that $H \in \supp(\phi_i)$ and $H \not\in \supp(\phi^{(n)}_i)$ for some $n$ and that there is no other monomial in $\supp(\phi)$ that not in $\supp(\phi^{(n)})$.  Then notice that for some $m > n$ we will also have $H \in \supp(\phi_i^{(m)})$ since $H$ is the product of other least monomials under iteration.  If it were true that $a_P(n) >a_p(1)$, then it would also be true that $a_P(m) < a_P(n)$ which contradicts Lemma \ref{lem10}.
        \end{proof}

        \begin{defn}
            Let $f \in K[[X_1,\ldots,X_h]]$ be a non-zero element.  Then we call $F
            \in \supp(f)$ a \emph{least monomial of $f$} if
            \begin{equation*}
                \{X^v \mid X^v \in \supp(f), v \neq 0 \text{, and } X^v \text{ divides } F \} = \{F\}
            \end{equation*}
            and $F$ is not the product of other least monomials in the support of $\phi$ under iteration.
        \end{defn}

        It is clear that one of the monomials in the monomial support of $\phi^n$ can be $0$
        in $K$ when the coefficient $\lambda_i$ of the linear term $x_i$ in $\supp(\phi_i)$ satisfies $\lambda_i \neq 1$ and $\lambda_i^n=1$ for some $1 \leq i \leq b$.  It is also possible to have a coefficient of $0$ after iteration when $\text{char } K = p$ and $n=Mp^e$ for some $e \geq 1$.
        Lemma \ref{lem11} gives general conditions for when the coefficient of a least monomial
        is divisible by $p$ after iteration.

        Denote $d\phi_P$ as the map induced by $\phi$ on the cotangent
        space of $X$ at $P$.  Recall that we are assuming that $P$ is a fixed point of $\phi$ and that $K$ is algebraically closed.  Therefore, $d\phi_P$ is a $b \times b$ matrix and can always be put in Jordan-canonical form, with Jordan blocks $J_1,\ldots,J_k$ of the form
        \begin{equation*}
            J_i = \begin{pmatrix}
                \lambda_i & 1 & 0 & 0 \\
                0 & \lambda_i & 1 & 0 \\
                \vdots & \ddots & \ddots & \vdots \\
                0 & \cdots & 0 & \lambda_i
                \end{pmatrix}.
        \end{equation*}
        Having a non-trivial Jordan block causes more complicated interaction between the
        $\phi_i$ through the additional linear terms.
        Consider as an example $J_1$ of size $\mathfrak{v}$.  We have
        \begin{align*}
            \phi_1 &= \lambda_1 x_1 + x_2 + \text{higher order terms} \\
            \phi_2 &= \lambda_1 x_2 + x_3 + \text{higher order terms} \\
            \vdots & = \vdots \\
            \phi_{\mathfrak{v}} &= \lambda_1 x_{\mathfrak{v}} + \text{higher order terms}.
        \end{align*}
        Along with the linear terms $\lambda_1x_1, \ldots, \lambda_1x_{\mathfrak{v}}$, we also have
        the linear terms $x_2,\ldots, x_{\mathfrak{v}}$.  Notice that for $F$, a least monomial
        in the monomial support of $\phi_i$ in a non-trivial Jordan block of $d\phi_P$, it may be that $F \in \supp(\phi_j)$
        for some other $\phi_j$ in the same Jordan block of $d\phi_P$.    We will be concerned
        with $F \not\in \supp(\phi^{(n)}_j)$ for every $\phi_j$ in the same Jordan block of $d\phi_P$.
        First we describe the coefficients of a least monomial under iteration.

        \begin{lem} \label{lem19}
            Let $F = \prod x_i^{e_i}$ and assume the monomial support of $\phi_{i_t}$
            with $1 \leq i_t \leq b$
            contains a least monomial that is a constant multiple of $F$.
            Assume that $\phi_{i_t}$ is in a Jordan block of $d\phi_P$ of size $\mathfrak{v} \geq 1$ with eigenvalue $\lambda$.  Label
            the rows $1, \ldots, \mathfrak{v}$ corresponding to $\phi_{i_1},\ldots,\phi_{i_{\mathfrak{v}}}$.
            Assume that the last row containing a constant multiple of $F$ is $\phi_{i_s}$ and label
            the initial coefficients of the $F$ as $c_1, \ldots, c_s$ with at least
            $c_s \neq 0$. Let $\alpha = \sum_{j=1}^{\mathfrak{v}} e_{i_j}$.
            Then we have that the coefficient of $F$ in $\phi_{i_t}^{(n)}$
            is determined as follows:
            \begin{enumerate}
                \item If $\deg{F} = 1$, then we have
                    \begin{equation*}
                        \begin{cases}
                            0 & \text{if } n < s-t,\\
                            \binom{n}{s-t} \lambda^{n-(s-t)} &
                            \text{otherwise}.
                        \end{cases}
                    \end{equation*}

                \item If $\deg{F} >1$, then
                    \begin{equation*}
                        \sum_{\ell=0}^{s-t}\left(\sum_{j=0}^{n-1-\ell} \binom{n-1-j}{\ell}
                        \lambda^{j(\alpha-1)+n-1-\ell}
                        \mathop{\prod_{x_i \mid F}}_{i \not\in \{i_1,\ldots,i_{\mathfrak{v}}\}}
                        \lambda_i^{e_i j}\right)c_{t+\ell}.
                    \end{equation*}
            \end{enumerate}
        \end{lem}
        \begin{proof}
            We will prove both statements by induction.

            \begin{enumerate}
                \item  For the base case of $n=1$, we have from the formula
                    \begin{equation*}
                        \begin{cases}
                        0 & \text{if } t < s-1, \\
                        1 & \text{if } t=s-1, \\
                        \lambda & \text{if } t=s,
                        \end{cases}
                    \end{equation*}
                    which corresponds to the linear terms of a
                    Jordan block.  We will assume now that the formula holds for the
                    $n$-th iterate and consider the $(n+1)$-st iterate.

\setcounter{case}{0}
                 \begin{case}   The Jordan block is size 1 or $t=s$.

                    In this case, the contribution to $F$ in $\supp(\phi_{i_s})$ through iteration  is given by
                    \begin{equation*}
                        \lambda x_s.
                    \end{equation*}
                    Hence, the coefficient of $F$ in $\supp(\phi_{i_s}^{(n+1)})$ is given by
                    \begin{equation*}
                        \lambda(\lambda^n x_s) = \lambda^{n+1}x_s,
                    \end{equation*}
                    confirming the formula.
                 \end{case}
                 \begin{case} The Jordan block is non-trivial and $t \neq s$.

                    In this case, the contribution to $F$ in $\supp(\phi_{i_t})$ through iteration is given by
                    \begin{equation*}
                        \lambda x_t + x_{t+1},
                    \end{equation*}
                    and hence the coefficient of $F$ in $\supp(\phi_{i_t}^{(n+1)})$ is given by
                    \begin{align*}
                        &\lambda\left(\binom{n}{s-t}\lambda^{n-(s-t)}\right) + \binom{n}{s-(t+1)}
                        \lambda^{n-(s-(t+1))} \\
                        &= \lambda^{n-(s-t)+1}\left(\binom{n}{s-t} +\binom{n}{s-t-1}\right) \\
                        &=\lambda^{n-(s-t)+1}\left(\frac{n(n-1)\cdots(s-t+1)}{(n-(s-t))!}
                        + \frac{n(n-1)\cdots(s-t+1)(s-t)}{(n-(s-t)+1)!}\right)\\
                        &=\lambda^{n-(s-t)+1}\left(\frac{n(n-1)\cdots(s-t+1)(n+1-(s-t))}{(n-(s-t))!(n+1-(s-t))}
                        + \frac{n(n-1)\cdots(s-t+1)(s-t)}{(n-(s-t)+1)!}\right)\\
                        &=\lambda^{n-(s-t)+1}\left(\frac{(n+1)n(n-1)\cdots(s-t+1)}{(n-(s-t)+1)!}\right)\\
                        &=\lambda^{n+1-(s-t)}\binom{n+1}{s-t},
                    \end{align*}
                    confirming the formula.
                 \end{case}

                \item
                    For the base case of $n=1$, the sum over $j$ has a term only when $\ell=0$ and,
                    in that case, we have $\ell=j=0$.  So the coefficient of
                    $c_{t+\ell}$ in $\phi_{i_t}$ is
                    \begin{equation*}
                        \begin{cases}
                        0 & \text{if } \ell \neq 0,\\
                        1=\lambda^{0}\mathop{\prod_{x_i \mid F}}_{i \not\in \{i_1,\ldots,i_{\mathfrak{v}}\}}
                        \lambda_i^{0} & \text{if } \ell = 0.
                        \end{cases}
                    \end{equation*}
                    We will assume now that the formula holds for the $n$-th iterate and consider
                    the $(n+1)$-st iterate.

\setcounter{case}{0}
                    \begin{case} The Jordan block is size 1 or $t=s$ (in other words, $\ell=0$).

                    In this case, the contribution to $F$ in $\supp(\phi_{i_s})$ through iteration is given by
                    \begin{equation*}
                        \lambda x_s + c_sF,
                    \end{equation*}
                    and, hence, the coefficient in $F$ in $\supp(\phi_{i_s}^{(n+1)})$ is given by
                    \begin{equation}\label{eq51}
                        \lambda\left(\sum_{j=0}^{n-1}
                        \lambda^{j(\alpha-1)+n-1}
                        \mathop{\prod_{x_i \mid F}}_{i \not\in \{i_1,\ldots,i_{\mathfrak{v}}\}}
                        \lambda_i^{e_i j}\right)
                        +(\lambda^{n})^{\alpha}
                        \mathop{\prod_{x_i \mid F}}_{i \not\in \{i_1,\ldots,i_{\mathfrak{v}}\}}
                        (\lambda_i^{n})^{e_i}.
                    \end{equation}
                    Notice that the first term of (\ref{eq51}) and the desired term of
                    \begin{equation}\label{eq52}
                        \sum_{j=0}^{n}
                        \lambda^{j(\alpha-1)+n}
                        \mathop{\prod_{x_i \mid F}}_{i \not\in \{i_1,\ldots,i_{\mathfrak{v}}\}}
                        \lambda_i^{e_i j}
                    \end{equation}
                    differ by exactly the $j=n$ term of (\ref{eq52}), which is exactly
                    the second term of (\ref{eq51}), confirming the formula.
                    \end{case}
                    \begin{case} The Jordan block is non-trivial and $t \neq s$.

                    In this case, the contribution to $F$ in $\supp(\phi_{i_t})$ through iteration is given by
                    \begin{equation} \label{eq1}
                        \lambda x_t + x_{t+1} + c_tF.
                    \end{equation}
\setcounter{subcase}{0}
                    \begin{subcase} $\ell=0$.

                    Since there is no contribution to $c_{t}$ from $x_{t+1}$ in (\ref{eq1}), the contribution
                    of $c_{t}$ is given by
                    \begin{equation} \label{eq53}
                        \lambda\left(\sum_{j=0}^{n-1}
                        \lambda^{j(\alpha-1)+n-1}
                        \mathop{\prod_{x_i \mid F}}_{i \not\in \{i_1,\ldots,i_{\mathfrak{v}}\}}
                        \lambda_i^{e_i j}\right)
                        +(\lambda^{n})^{\alpha}
                        \mathop{\prod_{x_i \mid F}}_{i \not\in \{i_1,\ldots,i_{\mathfrak{v}}\}}
                        (\lambda_i^{n})^{e_i}.
                    \end{equation}
                    Notice that the first term of (\ref{eq53}) and the desired term of
                    \begin{equation}\label{eq54}
                        \sum_{j=0}^{n}
                        \lambda^{j(\alpha-1)+n}
                        \mathop{\prod_{x_i \mid F}}_{i \not\in \{i_1,\ldots,i_{\mathfrak{v}}\}}
                        \lambda_i^{e_i j}
                    \end{equation}
                    differ by exactly the $j=n$ term of (\ref{eq54}), which is exactly
                    the second term of (\ref{eq53}), confirming the formula.
                    \end{subcase}
                    \begin{subcase} $\ell \neq 0$.

                    Since there is no contribution to $c_{t+\ell}$ from $c_tF$ in (\ref{eq1}) for
                    $\ell \neq 0$, the contribution of each $c_{t+\ell}$ is given by
                    \begin{align}
                        &\lambda\left(\sum_{j=0}^{n-1-\ell} \binom{n-1-j}{\ell}
                        \lambda^{j(\alpha-1)+n-1-\ell}
                        \mathop{\prod_{x_i \mid F}}_{i \not\in \{i_1,\ldots,i_{\mathfrak{v}}\}}
                        \lambda_i^{e_i j}\right) \label{eq55} \\
                        &+\sum_{j=0}^{n-1-\ell+1} \binom{n-1-j}{\ell-1}
                        \lambda^{j(\alpha-1)+n-1-\ell+1}
                        \mathop{\prod_{x_i \mid F}}_{i \not\in \{i_1,\ldots,i_{\mathfrak{v}}\}}
                        \lambda_i^{e_i j}.\label{eq56}
                    \end{align}
                    For $j = n-\ell$ there is no contribution from (\ref{eq55}),
                    so we have to check that
                    \begin{equation*}
                        \binom{n-j}{\ell} = \binom{n-1-j}{\ell-1}.
                    \end{equation*}
                    Computing, we get
                    \begin{equation*}
                        \binom{\ell}{\ell} = \binom{\ell-1}{\ell-1}
                    \end{equation*}
                    with equality since $\ell$ is at least 1.

                    We have left to check that for $j = 0, \ldots, n-\ell-1$ we have
                    \begin{equation*}
                        \binom{n-j}{\ell} = \binom{n-1-j}{\ell} + \binom{n-1-j}{\ell-1}.
                    \end{equation*}
                    Computing the right-hand side, we have
                    \begin{align*}
                        &\frac{(n-1-j)\cdots(\ell+1)}{(n-j-\ell-1)!} +
                        \frac{(n-1-j)\cdots(\ell)}{(n-j-\ell)!}\\
                        &=\frac{(n-1-j)\cdots(\ell+1)(n-j-\ell)}{(n-j-\ell)!} +
                        \frac{(n-1-j)\cdots(\ell+1)(\ell)}{(n-j-\ell)!}\\
                        &=\frac{(n-1-j)\cdots(\ell+1)(n-j-\ell+\ell)}{(n-j-\ell)!}\\
                        &=\frac{(n-j)(n-j-1)\cdots(\ell+1)}{(n-j-\ell)!}\\
                        &=\binom{n-j}{\ell},
                    \end{align*}
                    confirming the formula.
                \end{subcase}
                \end{case}
            \end{enumerate}
        \end{proof}

        \begin{rem}
            If $F \in \supp(\phi_i)$ with $\lambda_i=0$ then we know that $F$ does not effect $LT(I_1)$ since $x_i$ either divides $LT(f)$ or is relatively prime to $LT(f)$ for all $f \in I_1$.  In the former, case we take the normal form of $f$ with respect to the known leading terms.  In the latter case, we see that every term in the local analogue of the $S$-polynomials is divisible by the known leading terms and hence is already in the leading term ideal.

            If $F \in \supp(\phi_j)$ with $x_i \mid F$ and $\lambda_i=0$ then we know that $F$ does not effect $LT(I_1)$ since $x_i \in LT(I_n)$ for all $n$.

            So we will exclude from consideration the Jordan block(s) with eigenvalue $0$ and monomials divisible by $x_i$ with $\lambda_i=0$.
        \end{rem}

        \begin{lem} \label{lem11}
            Let $F = \prod x_i^{e_i}$ with $\deg{F} > 1$ and assume that
            $\supp(\phi_{i_t})$ with $1 \leq i_t \leq b$ contains a least monomial that is a constant multiple of $F$.  Assume that $\phi_{i_t}$ is in a Jordan block of $d\phi_P$ of size $\mathfrak{v} \geq 1$ with eigenvalue $\lambda \neq 0$.  Label the rows $1, \ldots, \mathfrak{v}$ corresponding to $\phi_{i_1},\ldots,\phi_{i_{\mathfrak{v}}}$. Let $\alpha = \sum_{j=1}^{\mathfrak{v}} e_{i_j}$.
            The following are conditions for all of the coefficients of $F$ in the Jordan block of $d\phi_P$ containing $F$ to be divisible by $p=\Char(K)$.

            \begin{enumerate}
                \item \label{lem11a} If $\lambda = 1$ and $\lambda_i=1$ for all $i$ such that $x_i \mid F$
                    and $i \not\in \{i_1,\ldots, i_{\mathfrak{v}}\}$, then it is necessary that $p \mid n$ and sufficient that $n=Mp^{e}$ for $e$ large enough and $(M,p)=1$.

                \item \label{lem11b} Assume $\lambda \neq 1$ and $\lambda_i = 1$ for all $i$
                    such that $x_i \mid F$ and $i \not\in \{i_1,\ldots, i_{\mathfrak{v}}\}$.

                    \begin{enumerate}
                        \item \label{lem11b1} If $\alpha =0$, then it is necessary that $\lambda$
                            is an $r$-th root of unity for some $r \mid n$ and sufficient that $n=Mrp^e$ for $e$ large enough.

                        \item \label{lem11b2} If $\alpha >0$, then it is necessary that
                            $\lambda^{\alpha-1}$ is an $r$-th root of unity for some $r \mid n$.
                    \end{enumerate}

                \item \label{lem11c} If $\lambda = 1$ and $\lambda_i \neq 1$ for at least one $i$
                    such that $x_i \mid F$, then it is necessary that
                    \begin{equation*}
                        \mathop{\prod_{x_i \mid F}}_{i \not\in \{i_1,\ldots,i_{\mathfrak{v}}\}}
                        \lambda_i^{e_i}
                    \end{equation*}
                    is an $r$-th root of unity for some $r \mid n$.

                \item \label{lem11d} If $\lambda \neq 1$ and $\lambda_i \neq 1$ for at least one $i$ such that $x_i \mid F$ and $i \not\in \{i_1,\ldots, i_{\mathfrak{v}}\}$, then it is necessary that
                    \begin{equation*}
                        \lambda^{\alpha-1}
                        \mathop{\prod_{x_i \mid F}}_{i \not\in \{i_1,\ldots,i_{\mathfrak{v}}\}}
                        \lambda_i^{e_i}
                    \end{equation*}
                    is an $r$-th root of unity with $r \mid n$.
            \end{enumerate}
           \end{lem}

        \begin{proof}
            We will use the description of the coefficients of $F$ under iteration from
            Lemma \ref{lem19}.  Assume that the last row containing $F$ is $\phi_{i_s}$ with
            $1 \leq s \leq \mathfrak{v}$ and label the initial coefficients of the $F$ as
            $c_1,\ldots, c_s$ with at least $c_s \neq 0$.

            If $\deg{F}=1$, then $F = x_{i_s}$ and $c_s = \lambda$ and
            the coefficient in $\phi^{(n)}_{i_s}$ is
            \begin{equation*}
                \lambda^{n}c_s.
            \end{equation*}
            Since $\lambda \neq 0$ in $K$ this coefficient is never divisible by $p$.  So we restrict to the case $\deg{F} >1$.

            \begin{enumerate}
                \item  We want the coefficients to be $0$ in $K$.
                    The coefficient in $\phi_{i_s}$ is given by
                    \begin{equation*}
                        \left(\sum_{j=0}^{n-1}1\right)c_s= nc_s
                    \end{equation*}
                    with $c_s \neq 0$ in $K$.  Hence, we must have $p \mid n$.
                    To see the sufficient condition; assume that $p \mid n$, then
                    at least $F \not\in \supp(\phi^{(n)}_{i_s})$.  Note that if $p \neq 2$, then
                    we also $F \not \in \supp(\phi^{(n)}_{i_{s-1}})$ since the coefficient
                    \begin{equation*}
                        \left(\sum_{j=1}^{n-1}j \right)c_{s-1} = \frac{n(n-1)}{2} c_{s-1}
                    \end{equation*}
                    is divisible by $p$ for $p \mid n$.  Now replace $\phi$ by $\phi^p$ and $n$ by $n/p$ and consider $F$ in the Jordan block of $d\phi_P$.
                    Now the last row containing $F$ is $\phi_{i_s}$.  By the above
                    argument, this coefficient in $\phi^p$ will be 0 in $K$.  Hence, with
                    each power of $p$, $F$ does not appear in the last previously
                    appearing row after iteration.  Since the Jordan block is of finite
                    size, taking $\phi^{p^e}$ for $e$ large enough causes the monomial $F$
                    to not appear in the Jordan block of $d\phi^n_P$.
                    So it is necessary that $p \mid n$ and sufficient that
                    $n=Mp^e$ for $e$ large enough and $(M,p)=1$.

                \item
                    \begin{enumerate}
                        \item
                            We want the coefficients to be $0$ in $K$.
                            The coefficient in $\phi_{i_s}$ is given by
                            \begin{equation*}
                                \left(\sum_{j=0}^{n-1} \lambda^{j}\right)c_{s}
                            \end{equation*}
                            with $c_s \neq 0$ in $K$.  Hence, we must have
                            \begin{equation*}
                                \lambda^n \equiv 1 \mod{p}.
                            \end{equation*}

                            To see the sufficient condition, assume
                            that $\lambda^r \equiv 1 \mod{p}$ for some $r \mid n$ and replace
                            $\phi$ by $\phi^r$ and $n$ by $n/r$.  Now we are in the situation of
                            (a), which we have already demonstrated.
                            So it is necessary that $\lambda$ is an $r$-th root of unity for some $r \mid n$ and sufficient that $n=Mrp^e$ for $e$ large enough and $(M,p)=1$.

                        \item
                            We want the coefficients to be $0$ in $K$.
                            The coefficient in $\phi_{i_s}$ is given by
                            \begin{equation*}
                                \left(\sum_{j=0}^{n-1} \lambda^{j(\alpha-1)+n-1}\right)c_{s}
                            \end{equation*}
                            with $c_s \neq 0$ in $K$.  Hence, $\lambda^{\alpha-1}$ must be an
                            $r$-th root of unity for some $r \mid n$.
                    \end{enumerate}

                \item We want the coefficients to be $0$ in $K$.
                    The coefficient in $\phi_{i_s}$ is given by
                    \begin{equation*}
                        \left(\sum_{j=0}^{n-1} \mathop{\prod_{x_i \mid F}}_{i \not\in \{i_1,\ldots,i_{\mathfrak{v}}\}}
                        \lambda_i^{e_i j}\right) c_{s}
                    \end{equation*}
                    with $c_s \neq 0$ in $K$.  Hence,
                    \begin{equation*}
                        \mathop{\prod_{x_i \mid F}}_{i \not\in \{i_1,\ldots,i_{\mathfrak{v}}\}}
                        \lambda_i^{e_i}
                    \end{equation*}
                    must be an $r$-th root of unity modulo $p$ for some $r \mid n$.

                \item We want the coefficients to be $0$ in $K$.
                    The coefficient in $\phi_{i_s}$ is given by
                    \begin{equation*}
                        \left(\sum_{j=0}^{n-1}\lambda^{j(\alpha-1)+n-1}
                        \mathop{\prod_{x_i \mid F}}_{i \not\in \{i_1,\ldots,i_{\mathfrak{v}}\}}
                        \lambda_i^{e_i j}\right)c_{s}
                    \end{equation*}
                    Hence,
                    \begin{equation*}
                        \lambda^{\alpha-1}
                        \mathop{\prod_{x_i \mid F}}_{i \not\in \{i_1,\ldots,i_{\mathfrak{v}}\}}
                        \lambda_i^{e_i}
                    \end{equation*}
                    must be an $r$-th root of unity modulo $p$ for some $r \mid n$.
            \end{enumerate}
        \end{proof}

       We have now established necessary conditions for a least monomial in
        \begin{equation*}
            \supp(\phi_1(\textbf{x})-x_1,\ldots,\phi_b(\textbf{x})-x_b)
        \end{equation*}
        to not appear in
        \begin{equation*}
            \supp(\phi_1^{(n)}(\textbf{x})-x_1,\ldots,\phi_b^{(n)}(\textbf{x})-x_b).
        \end{equation*}
        However, this vanishing is not sufficient for $a_P(n) \neq a_P(1)$.
        Fortunately, the necessary conditions on $n$ from Lemma \ref{lem11} will be enough to
        show that $\Phi^{\ast}_n(\phi)$ is an effective zero-cycle for all $n \geq 1$.

        The next proposition gathers our knowledge of $a_P(n)$.
        \begin{prop} \label{prop10}
            Let $X \subset \mathbb{P}_K^{N}$ be a non-singular, irreducible, projective variety
            of dimension $b$ defined over $K$.
            Let $\phi:X \to X$ be a morphism defined over $K$ and
            $P \in X(K)$ be a fixed point of $\phi$.  Denote $d\phi_P$ as the map induced by $\phi$ on the
            cotangent space of $X$ at $P$.  Let $\lambda_1,\ldots,\lambda_l$ be the
            distinct eigenvalues of $d\phi_P$ with primitive multiplicative orders
            $r_1,\ldots,r_l$ (set $r_i=\infty$ if $\lambda_i$ is not a root of
            unity). Then for all $n \geq 1$ such that $\phi^n$ is non-degenerate,
            \begin{enumerate}
                \item \label{prop10a} $a_P(n) \geq a_P(1)$.
                \item \label{prop10b} $a_P(n) = 1 \Leftrightarrow \lambda_i^n \neq 1$ for $1 \leq i \leq l$.
                \item \label{prop10c} If $a_P(n) > a_P(1)$, then at least one of the following is
                    true.
                    \begin{enumerate}
                        \item \label{prop10c1} $r_i \mid n$ for at least one $i$ for $1 \leq i \leq l$ with $r_i \neq 1$.
                        \item \label{prop10c2} $p= \text{char}(K) \neq 0$ and $p \mid n$.
                    \end{enumerate}
            \end{enumerate}
        \end{prop}

        \begin{proof}
            \mbox{}
            \begin{enumerate}
            \item Lemma \ref{lem10}.

            \item
                It is clear that $a_P(n)=1$ if and only if $I_n$
                generates the maximal ideal of $\widehat{R_P}$.  This is true if and only if
                \begin{equation*}
                    \{x_1-y_1,\ldots,x_b-y_b,\phi^{(n)}_1(\textbf{x})-y_1,\ldots,
                    \phi^{(n)}_b(\textbf{x})-y_b\}
                \end{equation*}
                is a regular local system of parameters.
                Zariski and Samuel \cite[Corollary 2 page 137]{zariski2} state that this
                occurs if and only if the power series
                \begin{equation*}
                    \{x_1-y_1,\ldots,x_b-y_b,\phi^{(n)}_1(\textbf{x})-y_1,\ldots,
                    \phi^{(n)}_b(\textbf{x})-y_b\}
                \end{equation*}
                contain independent linear terms.
                This is true if and only if $\lambda_i^n \neq 1$ for $1 \leq i \leq l$.

            \item
                We know from Corollary \ref{cor9} that $a_P(n)>a_P(1)$ if and only if certain monomials $F$ has zero coefficients after iteration.  Any such monomial must be a least monomial by Lemma \ref{lem8} and Lemma \ref{lem1}.  Lemma \ref{lem11} gives necessary conditions on $n$ for which any least monomial has zero coefficients after iteration.  Note that cases (\ref{lem11b2}), (\ref{lem11c}), and (\ref{lem11d}) of Lemma \ref{lem11} are cases where $a_P(n) = a_P(1)$ since $\lambda_i \neq 1$ for some $x_i \mid F$.  Hence, the removal of this monomial has no effect on the leading term ideal.  So we are concerned only with the conditions (\ref{lem11a}) and (\ref{lem11b1}) of Lemma \ref{lem11} for which we also know sufficient conditions.
            \end{enumerate}
        \end{proof}

\subsection{Proof of effectivity}
        We will consider several different maps over the course of the proof, so to
        avoid confusion we will include the map in the notation as $a_{P}(\phi,n)$ and
        $a_P^{\ast}(\phi,n)$.

        \begin{lem} \label{lem9}
            Let $p$ be a prime in $\mathbb{Z}$ and let $n = Mp^{e}$ in $\mathbb{Z}^{+}$
            with $e \geq 1$ and $p \nmid M$.
            \begin{enumerate}
                \item \label{lem9a} If $e=1$, then
                    \begin{equation*}
                        a_P^{\ast}(\phi,n) = a_P^{\ast}(\phi^p,M) -
                        a_P^{\ast}(\phi,M).
                    \end{equation*}
                \item \label{lem9b} If $e \geq 2$, then
                    \begin{equation*}
                        a_P^{\ast}(\phi,n) =
                        a_P^{\ast}(\phi^{p^{e-1}},Mp).
                    \end{equation*}
                \item \label{lem9c} Let $n=qM$ where $\gcd(q,M)=1$.  Then
                    \begin{equation*}
                        a_P^{\ast}(\phi,n) = \sum_{d \mid q} \mu\left(\frac{q}{d}\right)a_P^{\ast}(\phi^d,M).
                    \end{equation*}
            \end{enumerate}
        \end{lem}

        \begin{proof}
            Computing, we get
            \begin{align*}
                a_P^{\ast}(\phi,n) &= \sum_{d \mid n}
                \mu\left(\frac{n}{d}\right)a_P(\phi,d)\\
                &= \sum_{pd \mid n} \mu\left(\frac{n}{pd}\right)a_P(\phi,pd) +
                    \sum_{d \mid M} \mu\left(\frac{n}{d}\right)a_P(\phi,d) \\
                &= \sum_{d \mid Mp^{e-1}} \mu\left(\frac{Mp^{e-1}}{d}\right)a_P(\phi^p,d) +
                    \sum_{d \mid M} \mu\left(\frac{Mp^e}{d}\right)a_P(\phi,d) \\
                &= a_P^{\ast}(\phi^p,Mp^{e-1}) +
                    \sum_{d \mid M} \mu\left(\frac{Mp^e}{d}\right)a_P(\phi,d).
            \end{align*}
            So we have
            \begin{equation} \label{eq4}
                a_P^{\ast}(\phi,n) =
                a_P^{\ast}(\phi^p,Mp^{e-1}) +
                    \sum_{d \mid M} \mu\left(\frac{Mp^e}{d}\right)a_P(\phi,d).
            \end{equation}

            \begin{enumerate}
                \item Considering (\ref{eq4}) with $e=1$, we have
                    \begin{align*}
                        a_P^{\ast}(\phi,n) &= a_P^{\ast}(\phi^p,M) +
                        \sum_{d \mid M} \mu\left(\frac{Mp}{d}\right)a_P(\phi,d) \\
                        &= a_P^{\ast}(\phi^p,M) +
                            \sum_{d \mid M} \mu(p) \mu\left(\frac{M}{d}\right)a_P(\phi,d) \\
                        &= a_P^{\ast}(\phi^p,M) - a_P^{\ast}(\phi,M),
                    \end{align*}
                    where the middle equality comes from the fact that $\mu$ is multiplicative and
                        $(p,M)=1$.

                \item Considering (\ref{eq4}) with $e>1$, we have
                    \begin{align*}
                        a_P^{\ast}(\phi,n) &= a_P^{\ast}(\phi^p,Mp^{e-1}) +
                        \sum_{d \mid M} \mu\left(\frac{Mp^e}{d}\right)a_P(\phi,d) \\
                        &= a_P^{\ast}(\phi^p,Mp^{e-1}) + 0,
                    \end{align*}
                    where the second equality comes from the fact that $\frac{Mp^e}{d}$
                    is not square free for all $d \mid M$.
                    Replacing $\phi$ by $\phi^p$ and $n$ by $n/p$, we may repeat the
                    argument to conclude that
                    \begin{equation*}
                        a_P^{\ast}(\phi,n) =
                        a_P^{\ast}(\phi^{p^{e-1}},Mp).
                    \end{equation*}

                \item  Using the fact that the M\"{o}bius function is multiplicative for relatively prime
                    numbers, we get
                    \begin{align*}
                        a_P^{\ast}(\phi,n) &= \sum_{d \mid qM} \mu\left(\frac{qM}{d}\right)a_P(\phi,d) \\
                        &=\sum_{d_1 \mid q}\sum_{d_2 \mid M}
                        \mu\left(\frac{qM}{d_1d_2}\right)a_P(\phi,d_1d_2)\\
                        &=\sum_{d_1 \mid q}\sum_{d_2 \mid M}
                        \mu\left(\frac{q}{d_1}\right)\mu\left(\frac{M}{d_2}\right)a_P(\phi,d_1d_2)\\
                        &=\sum_{d_1 \mid q}\mu\left(\frac{q}{d_1}\right)\sum_{d_2 \mid M}
                        \mu\left(\frac{M}{d_2}\right)a_P(\phi^{d_1},d_2)\\
                        &=\sum_{d_1 \mid
                        q}\mu\left(\frac{q}{d_1}\right)a_P^{\ast}(\phi^{d_1},M).
                    \end{align*}
            \end{enumerate}
        \end{proof}

        In the next lemma, we provide a formula for $a_P^{\ast}(n)$ when
        $n=\lcm(r_{i_1},\ldots,r_{i_k})$ for some subset $\{r_{i_1},\ldots,r_{i_k}\}$
        of $\{r_1,\ldots,r_l\}$.  We clearly need that each $r_{i_t}$ is finite,
        in other words, that $\lambda_{i_t}$ has finite order, and we will also assume that each
        $r_{i_t} \neq 1$.
        \begin{lem} \label{lem13}
            Let $P$ be a fixed point of $\phi$.  Let $r_i$ be the primitive order of
            $\lambda_i$ in $K^{\ast}$ for $1 \leq i \leq l$ (set $r_i=\infty$ if $\lambda_i$ is not a root of unity).
            If $n=\lcm(r_{i_1},\ldots,r_{i_k})$ for some subset of non-trivial finite orders
            $\{r_{i_1},\ldots,r_{i_k}\} \subseteq \{r_1,\ldots,r_l\}$ with $n \neq 0$ in $K$ and
            square free with no other $r_i$ dividing $n$, then we have
            \begin{align*}
                a_P^{\ast}(\phi,n) &= \\
                &\sum_{d\mid n} \mu\left(\frac{n}{d}\right)a_P(1) \\
                &+\sum_{t=1}^{k} \sum_{d \mid \frac{n}{r_{i_t}}}
                \mu\left(\frac{n}{dr_{i_t}}\right) c_{i_t} \\
                &+\sum_{t1=1}^{k} \mathop{\sum_{t2=1}^{k}}_{i_{t2} \neq i_{t1}}
                \sum_{d \mid \frac{n}{\lcm(r_{i_{t1}},r_{i_{t2}})}}
                \mu\left(\frac{n}{d\lcm(r_{i_{t1}},r_{i_{t2}})}\right) c_{i_{t1},i_{t2}} \\
                & \vdots \\
                &+ \sum_{t=1}^{k} \sum_{d \mid \frac{n}
                {\lcm(r_{i_1},\ldots,\widehat{r_{i_t}},\ldots,r_{i_k})}}
                \mu\left(\frac{n}{d\lcm(r_{i_1},\ldots,\widehat{r_{i_t}},\ldots,r_{i_k})}\right)
                c_{{i_1},\ldots,\widehat{c_{i_t}},\ldots,i_k}\\
                &+\sum_{d \mid 1} c_{i_1,\cdots,i_k}
            \end{align*}
            for some non-negative constants $c_{\alpha}$.
        \end{lem}

        \begin{proof}
            Recall that
            \begin{equation*}
                LT(I_{\Gamma_n} + I_{\Delta})=LT(I_n) \subseteq LT(I_1)=LT(I_{\Gamma_1} + I_{\Delta}).
            \end{equation*}
            In particular, we know that
            \begin{equation*}
                \Span(X^{v} \mid X^{v} \not\in LT(I_1))
                \subseteq
                \Span(X^{v} \mid X^{v} \not\in LT(I_n)).
            \end{equation*}

            From Proposition \ref{prop10}, we know that $a_P(r_{i_t}) > a_P(1)$ for each
            $r_{i_t}$ since $r_{i_t}\neq 1$.  Similarly for $i_{t1} \neq i_{t2}$,
            by replacing $\phi$ with $\phi^{r_{i_{t1}}}$,
            we have
            \begin{equation*}
                \begin{cases}
                    a_P(\lcm(r_{i_{t1}},r_{i_{t2}})) > a_P(r_{i_{t1}}) & \text{if } r_{i_{t2}} \nmid r_{i_{t1}}\\
                    a_P(\lcm(r_{i_{t1}},r_{i_{t2}})) = a_P(r_{i_{t1}}) & \text{if } r_{i_{t2}} \mid
                    r_{i_{t1}}
                \end{cases}
            \end{equation*}
            since in the second case $\lcm(r_{i_{t1}},r_{i_{t2}}) = r_{i_{t1}}$.
            Continuing in the same manner, we have
            \begin{equation*}
                \begin{cases}
                    a_P(\lcm(r_{i_1},\ldots,r_{i_j},r_{i_{\gamma}})) > a_P(\lcm(r_{i_1},\ldots,r_{i_j}))
                    & \text{if } r_{i_{{\gamma}}} \nmid \lcm(r_{i_1},\ldots,r_{i_j})\\
                    a_P(\lcm(r_{i_1},\ldots,r_{i_j},r_{i_{{\gamma}}})) = a_P(\lcm(r_{i_1},\ldots,r_{i_j}))
                    & \text{if } r_{i_{\gamma}} \mid \lcm(r_{i_1},\ldots,r_{i_j},r_{i_{{\gamma}}}).
                \end{cases}
            \end{equation*}
            Again in the second case, we have $\lcm(r_{i_1},\ldots,r_{i_j},r_{i_{{\gamma}}}) =
            \lcm(r_{i_1},\ldots,r_{i_j})$, so we have left to consider the first case.
            In particular, for any $\beta$ defined as the least common multiple of any $j$ of
            $\{r_{i_1},\ldots,r_{i_j},r_{i_{{\gamma}}}\}$, we have
            $\{X^{v} \mid X^{v} \not\in LT(I_{\lcm(r_{i_1},\ldots,r_{i_j},r_{i_{{\gamma}}})})\}$
            containing at least one element not in $\{X^{v} \mid X^{v} \not\in
            LT(I_{\beta})\}$.  To see this, consider the ordering
            \begin{equation*}
                x_{i_1} < x_{i_2} < \cdots < x_{i_j} < x_{{\gamma}} < x_{i_t} < \cdots < x_{i_{b-j-1}}.
            \end{equation*}
            Then for each $\beta$, one of the linear terms $x_{i_1}, \ldots,
            x_{i_j},x_{i_{{\gamma}}}$ is contained in $LT(I_{\beta})$ since it is a leading term
            of the associated $\phi_i^{(\beta)}(x_1,\ldots,x_b)-x_i$.  Also, none of the linear terms
            $x_{i_1}, \ldots, x_{i_j},x_{i_{{\gamma}}}$ are contained in
            $LT(I_{\lcm(r_{i_1},\ldots,r_{i_j},r_{i_{{\gamma}}})})$.  Hence, the monomial
            \begin{equation*}
                x_{i_1}x_{i_2}\cdots x_{i_j} x_{\gamma}
            \end{equation*}
            is in
            $\{X^{v} \mid X^{v} \not\in LT(I_{\lcm(r_{i_1},\ldots,r_{i_j},r_{i_{{\gamma}}})})\}$
            but not in any of the $\{X^{v} \mid X^{v} \not\in LT(I_{\beta})\}$.  This
            argument ensures the non-negativity of the constants $c_{\alpha}$ defined below.

            We have $a_P(1) \geq 1$ since $P$ is a fixed point and since
            \begin{equation*}
                \{X^{v} \mid X^{v} \not\in LT(I_1)\}
                \subseteq
                \{X^{v} \mid X^{v} \not\in LT(I_{\kappa})\}
            \end{equation*}
            for all $\kappa \geq 1$, we have a contribution of $a_P(1)$ to $a_P(d)$ for all
            $d \mid n$.

            Let $c_{i_t} = a_P(r_{i_t})-a_P(1) > 0$ for $1 \leq t \leq k$ since $r_{i_t} \neq 1$
            by assumption for $1 \leq t \leq k$.  Since
            \begin{equation*}
                \{X^{v} \mid X^{v} \not\in LT(I_{r_{i_t}})\}
                \subseteq
                \{X^{v} \mid X^{v} \not\in LT(I_{\kappa})\}
            \end{equation*}
            for all $\kappa$ with $r_{i_t} \mid \kappa$, we have a contribution of $c_{i_t}$ to $a_P(d)$ for all
            $d \mid \frac{n}{r_{i_t}}$.

            Let
            \begin{equation*}
                c_{i_{t1},i_{t2}} = a_P(\lcm(r_{i_{t1}},r_{i_{t2}})) - \#\{X^{v} \mid X^{v} \not\in LT(I_{r_{i_{t1}}})\}
                \cup \{X^{v} \mid X^{v} \not\in LT(I_{r_{i_{t2}}})\}.
            \end{equation*}
            If $r_{i_{t2}} \mid r_{i_{t1}}$, then $c_{i_{t1},i_{t2}}=0$ since
            $\lcm(r_{i_{t1}},r_{i_{t2}}) = r_{i_{t1}}$.  Otherwise, by the argument at the beginning of the proof, there is at least one monomial
            not in $LT(I_{r_{t1},r_{t2}})$ that is not in the complement of
            $LT(I_{r_{t1}})$ or $LT(I_{r_{t2}})$.  Hence $c_{i_{t1},i_{t2}} \geq 0$.  Since
            \begin{equation*}
                \{X^{v} \mid X^{v} \not\in LT(I_{\lcm(r_{i_{t1}},r_{i_{t2}})})\}
                \subseteq
                \{X^{v} \mid X^{v} \not\in LT(I_{\kappa})\}
            \end{equation*}
            for all $\kappa$ with $\lcm(r_{i_{t1}},r_{i_{t2}}) \mid \kappa$, we have a contribution of
            $c_{i_{t1},i_{t2}}$ to $a_P(d)$ for all $d \mid \frac{n}{\lcm(r_{i_{t1}},r_{i_{t2}})}$.

            Similarly, for $2 \leq j \leq k$, let $\beta$ be the least common multiple of $j$ elements of
            $\{r_{i_{t1}},\ldots,r_{i_{tj}},r_{i_{\gamma}}\}$ and let
            \begin{equation*}
                c_{i_{t1},\ldots,i_{tj},i_{\gamma}} = a_P(\lcm(r_{i_{t1}},\ldots,r_{i_{tj}},r_{i_{\gamma}})) -
                \#\left(\bigcup_{\beta}
                \{X^{v} \mid X^{v} \not\in LT(I_{\beta})\}\right).
            \end{equation*}
            If $r_{i_{\gamma}} \mid \lcm(r_{i_{t1}},\ldots,r_{i_{tj}})$, then
            $c_{i_{t1},\ldots,i_{tj},i_{{\gamma}}}=0$ since
            $\lcm(r_{i_{t1}},\ldots,r_{i_{tj}},r_{i_{{\gamma}}}) =
            \lcm(r_{i_{t1}},\ldots,r_{i_{tj}})$.
            Otherwise, by the argument at the beginning of the proof, there is at least one monomial not in $LT(I_{r_{i_{t1}},\ldots,r_{i_{tj}},r_{i_{{\gamma}}}})$ that is not in the complement of $LT(I_{\beta})$ for each $\beta$.
            Hence $c_{i_{t1},\ldots,i_{tj},i_{\gamma}} \geq 0$.
            Since
            \begin{equation*}
                \{X^{v} \mid X^{v} \not\in LT(I_{\lcm(r_{i_{t1}},\ldots,r_{i_{tj}},r_{i_{\gamma}})})\}
                \subseteq
                \{X^{v} \mid X^{v} \not\in LT(I_{\kappa})\}
            \end{equation*}
            for all $\kappa$ with $\lcm(r_{i_{t1}},\ldots,r_{i_{tj}},r_{i_{{\gamma}}}) \mid \kappa$,
            we have a contribution of
            $c_{i_{t1},\ldots,i_{tj},i_{\gamma}}$ to $a_P(d)$ for all
            $d \mid \frac{n}{\lcm(r_{i_{t1}},\ldots,r_{i_{tj}},r_{i_{{\gamma}}})}$.

            Notice that by construction, none of the monomials in
            $\{X^{v} \mid X^{v} \not\in LT(I_{n})\}$ are counted in multiple constants $c_{\alpha}$,
            and all of them have been counted.  Hence, the formula holds.
        \end{proof}

        \begin{rem}
            Notice that Lemma \ref{lem13} implies that
            $a_P^{\ast}(n) \geq 0$ for all $n=\lcm(r_{i_1},\ldots,r_{i_k})$ since each line is either
            $0$ or $c_{\alpha}$ by properties of the M\"obius function and
            the constants $c_{\alpha}$ are all non-negative.  We may
            assume $n$ is square free by Lemma \ref{lem9}(\ref{lem9b}).
        \end{rem}

        We are now ready to prove the main theorem.
        \begin{thm} \label{thm15}
            Let $X \subset \mathbb{P}_K^{N}$ be a non-singular, irreducible, projective variety
            of dimension $b$ defined over $K$ and let $\phi:X \to X$ be a morphism defined over $K$.
            Let $P$ be a point in $X(K)$. Define integers
            \begin{enumerate}
                \item[] $p =$ the characteristic of $K$.
                \item[] $m =$ the primitive period of $P$ for $\phi$ (set $m=\infty$ if $P \not\in
                \text{Per}(\phi)$).
            \end{enumerate}
            If $m$ is finite,
            let $d\phi^m_P$ be the map induced by $\phi^m$ on the
            cotangent space of $X$ at $P$, and
            let $\lambda_1, \ldots,\lambda_l$ be the distinct eigenvalues of $d\phi^m_P$.
            Define
            \begin{enumerate}
                \item[] $r_i =$ the multiplicative period of $\lambda_i$ in
                $K^{\ast}$ (set $r_i=\infty$ if $\lambda_i$ is not a root of
                unity).
            \end{enumerate}
            Then
            \begin{enumerate}
                \item For all $n \geq 1$ such that $\phi^n$ is non-degenerate, $a_P^{\ast}(n) \geq 0$.
                \item Let $n \geq 1$.  If $\phi^n$ is non-degenerate and $a_P^{\ast}(n) \geq 1$, then
                    $n$ has one of the following forms:
                    \begin{enumerate}
                        \item $n=m$.
                        \item $n=m\lcm(r_{i_1},\ldots,r_{i_k})$ for $1\leq k \leq l$.
                        \item $n=m\lcm(r_{i_1},\ldots,r_{i_k})p^e$ for $1\leq k \leq l$ and some $e \geq 1$.
                    \end{enumerate}
            \end{enumerate}
        \end{thm}

        \begin{proof}
            Fix a point $P \in X$ and let $n \geq 1$ be an integer such that $\phi^n$ is
            non-degenerate.  By definition, we have
            \begin{equation*}
                a_P^{\ast}(\phi,n) = \sum_{d \mid n} \mu\left(\frac{n}{d}\right) a_P(\phi,d).
            \end{equation*}
            Suppose that $\phi^{n}(P) \neq P$.  Then $\phi^{d}(P) \neq P$ for all $d \mid
            n$, so $a_P(\phi,d)=0$ for all $d \mid n$ since the graph $\Gamma_d$ of $\phi^d$ and
            the diagonal $\Delta$ will not intersect at $(P,P)$.  Hence, $a_P^{\ast}(\phi,n) =
            0$, proving the theorem in this situation.
            We now assume that $\phi^{n}(P) = P$.

            It follows that $P$ is a periodic point for $\phi$, so $m$ is finite
            with $m \mid n$ and $a_P(\phi,d) \geq 1$ if and only if $m \mid d$.
            Computing $a_P^{\ast}(\phi,n)$ in terms of $\phi^m$, we see that
            \begin{align*}
                a_P^{\ast}(\phi,n) &= \sum_{d \mid n \text{ with } m \mid d}
                \mu\left(\frac{n}{d}\right) a_P(\phi,d) \\
                &= \sum_{d\mid(n/m)} \mu\left(\frac{n}{md}\right)a_P(\phi,md) \\
                &= \sum_{d \mid (n/m)} \mu\left(\frac{n/m}{d}\right)a_P(\phi^m,d)\\
                &= a_P^{\ast}(\phi^m,n/m).
            \end{align*}
            Therefore, we can replace $\phi$ by $\phi^m$ and $n$ by $n/m$ and assume that $m=1$.

            We will consider a number of cases, but first we recall from
            Proposition \ref{prop10} that $a_P(\phi,1)=1$ if and only if $r_i \neq 1$ for all
            $1 \leq i \leq l$.

\setcounter{case}{0}
            \begin{case} $n=1$, in other words $n=m$.

                In this case, we have
                \begin{equation*}
                    a_{P}^{\ast}(\phi,n) = a_P(\phi,1).
                \end{equation*}
                Since $P$ is assumed to be fixed by $\phi$,
                \begin{equation*}
                    a_P^{\ast}(\phi,n) = a_P(\phi,1) \geq \begin{cases} 1 & \text{always} \\
                    2 & \text{ if } r_i = 1 \text{ for some } i. \end{cases}
                \end{equation*}
            \end{case}
            \begin{case} $n>1$ and $a_P(\phi,n)=a_P(\phi,1)$.

                Let $d \mid n$; then Proposition \ref{prop10} states that
                \begin{equation*}
                    a_P(\phi,1) = a_P(\phi,n) = a_P(\phi^d,n/d) \geq a_P(\phi^d,1) = a_P(\phi,d)
                    \geq a_P(\phi,1).
                \end{equation*}
                Hence, $a_P(\phi,d)=a_P(\phi,1)$ for all $d \mid n$.  So
                \begin{equation*}
                    a_P^{\ast}(\phi,n) = \sum_{d \mid n}\mu\left(\frac{n}{d}\right)a_P(\phi,1)=0
                \end{equation*}
                by properties of the M\"obius function, since $n>1$ by assumption.
            \end{case}
            \begin{case} $a_P(\phi,n)>a_P(\phi,1)$ and $n\neq 0$ in $K$.

                By the assumptions in this case, we know that at least one $r_{i} \mid n$.  Let
                $n=\lcm(r_{i_1},\ldots,r_{i_k})M$ where $M$ is not divisible by any
                $r_i$.  Then we have
                \begin{equation*}
                    a_P^{\ast}(\phi,n) = \sum_{d\mid M}a_P^{\ast}(\phi^{d},\lcm(r_{i_1},\ldots,r_{i_k}))
                \end{equation*}
                by Lemma \ref{lem9}(\ref{lem9c}).
                However, since $r_i \nmid M$ for all $1\leq i \leq l$, for $d \mid M$, we also have
                $r_i \nmid d$ for all $1 \leq i \leq l$.   Additionally, $n \neq 0$ implies
                $p \nmid n$, so we cannot be in any condition of Proposition
                \ref{prop10}(\ref{prop10c}).  Consequently,
                \begin{align*}
                    a_P^{\ast}(\phi,n) &= \sum_{d\mid M}\mu\left(\frac{M}{d}\right)a_P^{\ast}(\phi^{d},
                        \lcm(r_{i_1},\ldots,r_{i_k}))\\
                        &=\sum_{d\mid
                        M}\mu\left(\frac{M}{d}\right)a_P^{\ast}(\phi,\lcm(r_{i_1},\ldots,r_{i_k}))\\
                        &=0
                \end{align*}
                since $a_P^{\ast}(\phi^d,\lcm(r_{i_1},\ldots,r_{i_k}))$ is constant
                over $d\mid M$.
                So we can assume that $n=\lcm(r_{i_1},\ldots,r_{i_k})$ and $r_i \nmid n$
                for
                $i \not \in \{i_{1},\ldots,i_k\}$.  If $n$ is not square free, then by
                applying Lemma \ref{lem9} to any prime factor $q_j^{e_j}$ with
                $e_{j}>1$, we get
                \begin{equation*}
                    a_P^{\ast}(\phi,n) =
                    a_P^{\ast}\left(\phi^{q_1^{e_1-1}\cdots q_{j}^{e_j-1}},
                        \frac{n}{q_1^{e_1-1}\cdots q_{j}^{e_j-1}}\right).
                \end{equation*}
                So we may replace $n$ by
                $\frac{n}{q_1^{e_1-1}\cdots q_{j}^{e_j-1}}$ and $\phi$ by
                $\phi^{q_1^{e_1-1}\cdots q_{j}^{e_j-1}}$ and assume that $n$ is square free.
                We are now in the case of Lemma \ref{lem13} and have
                \begin{align*}
                    a_P^{\ast}(\phi,n) &=\sum_{d\mid n} \mu\left(\frac{n}{d}\right)a_P(1) \\
                    &+\sum_{t=1}^{k} \sum_{d \mid \frac{n}{r_{i_t}}}
                    \mu\left(\frac{n}{dr_{i_t}}\right) c_{i_t} \\
                    &+\sum_{t1=1}^{k} \mathop{\sum_{t2=1}^{k}}_{i_{t2} \neq i_{t1}}
                    \sum_{d \mid \frac{n}{\lcm(r_{i_{t1}},r_{i_{t2}})}}
                    \mu\left(\frac{n}{d\lcm(r_{i_{t1}},r_{i_{t2}})}\right) c_{i_{t1},i_{t2}} \\
                    & \vdots \\
                    &+ \sum_{t=1}^{k} \sum_{d \mid \frac{n}
                    {\lcm(r_{i_1},\ldots,\widehat{r_{i_t}},\ldots,r_{i_k})}}
                    \mu\left(\frac{n}{d\lcm(r_{i_1},\ldots,\widehat{r_{i_t}},\ldots,r_{i_k})}\right)
                    c_{{i_1},\ldots,\widehat{c_{i_t}},\ldots,i_k}\\
                    &+\sum_{d \mid 1} c_{i_1,\cdots,i_k}
                \end{align*}
                for some non-negative constants $c_{\alpha}$.
                Since every inner sum is either $0$ or $c_{\alpha}$ by properties of the
                M\"obius function, we have that $a_P^{\ast}(\phi,n) \geq 0$ because every
                $c_{\alpha}$ is non-negative.  By assumption, at least one $r_i$ divides $n$, so we know that $c_{n}$ will be positive since it will have at least one additional monomial.  Additionally, the sum associated to $c_n$ will be $c_{n}$ since it is summing over the divisors of $1$.
                So we have shown that
                \begin{equation*}
                    \begin{cases} a_P^{\ast}(\phi,n) \geq 1 & \text{ if M=1} \\
                    a_P^{\ast}(\phi,n) =0 &\text{ otherwise.}
                \end{cases}
                \end{equation*}
            \end{case}
            \begin{case} $a_P(\phi,n)>a_P(\phi,1)$ and $n=0$ in $K$.

                We can write $n=\lcm(r_{i_1},\ldots,r_{i_k})p^eM$ with
                $(r_{i_t},M)=1=(p,M)$ by Proposition \ref{prop10}.
                If $M>1$, then
                \begin{equation*}
                    a^{\ast}_P(\phi, \lcm(r_{i_1},\ldots,r_{i_k})p^e) =
                    a^{\ast}_P(\phi,d \lcm(r_{i_1},\ldots,r_{i_k})p^e)
                \end{equation*}
                for all $d \mid M$ since $M$ is not in one of
                the forms of Proposition \ref{prop10}(\ref{prop10c}).  So
                \begin{align*}
                    a_P^{\ast}(\phi,n)&=\sum_{d\mid M}
                    \mu\left(\frac{n}{d}\right) a_P(\phi^d, \lcm(r_{i_1},\ldots,r_{i_k})p^e)\\
                    &=\sum_{d\mid M}
                    \mu\left(\frac{n}{d}\right) a_P(\phi, \lcm(r_{i_1},\ldots,r_{i_k})p^e)\\
                    &=0,
                \end{align*}
                where the first equality is from Lemma \ref{lem9}(\ref{lem9c}).
                So assume $M=1$.  Computing, we have
                \begin{align*}
                    a_P^{\ast}(\phi,\lcm(r_{i_1},\ldots,r_{i_k}) p^e) &=
                    a_P^{\ast}(\phi^{p^{e-1}},\lcm(r_{i_1},\ldots,r_{i_k}) p) \\
                    &=a_P^{\ast}(\phi^{p^{e-1}},\lcm(r_{i_1},\ldots,r_{i_k}) p) -
                        a_P^{\ast}(\phi^{p^{e-1}},\lcm(r_{i_1},\ldots,r_{i_k}))\\
                    &=a_P^{\ast}(\phi^{p^e},\lcm(r_{i_1},\ldots,r_{i_k})) -
                        a_P^{\ast}(\phi^{p^{e-1}},\lcm(r_{i_1},\ldots,r_{i_k})).
                \end{align*}
                Considering the maps $\phi^{p^e}$ and $\phi^{p^{e-1}}$, we have have
                $\lcm(r_{i_1},\ldots,r_{i_k}) \neq 0$ in $K$.  As in Case 3,
                we may assume that $\lcm(r_{i_1},\ldots,r_{i_k})$ is square free and
                use Lemma \ref{lem13} to write $a_P^{\ast}(\lcm(r_{i_1},\ldots,r_{i_k}))$ in terms of
                the non-negative constants $c_{\alpha}$.  Since we are working with
                constants $c_{\alpha}$ for different maps, we include the map in the
                notation as $c_{\alpha}(\phi^{p^e})$.  The constants that contribute to
                $a_P^{\ast}(\lcm(r_{i_1},\ldots,r_{i_k}))$
                are associated to $\alpha = \lcm(r_{i_1},\ldots,r_{i_k})$ since the M\"{o}bius sum is not identically $0$ in that case.  So if
                \begin{equation*}
                    a_P(\phi,\lcm(r_{i_1},\ldots,r_{i_k}) p^e) =
                    a_P(\phi, \lcm(r_{i_1},\ldots,r_{i_k}) p^{e-1}),
                \end{equation*}
                then $c_{\alpha}(\phi^{p^e}) = c_{\alpha}(\phi^{p^{e-1}})$
                If we get additional key monomials with zero coefficients after iteration, in other words,
                \begin{equation*}
                    a_P(\phi,\lcm(r_{i_1},\ldots,r_{i_k}) p^e)
                    > a_P(\phi, \lcm(r_{i_1},\ldots,r_{i_k}) p^{e-1}),
                \end{equation*}
                then $c_{\alpha}(\phi^{p^e}) > c_{\alpha}(\phi^{p^{e-1}})$.
                Hence,
                \begin{equation*}
                \begin{cases}
                  a_P^{\ast}(\phi,n) = 0 &\text{if } a_P(\phi,\lcm(r_{i_1},\ldots,r_{i_k}) p^e) =
                    a_P(\phi, \lcm(r_{i_1},\ldots,r_{i_k}) p^{e-1}) \\
                  a_P^{\ast}(\phi,n) > 0 & \text{if } a_P(\phi,\lcm(r_{i_1},\ldots,r_{i_k}) p^e)
                  > a_P(\phi, \lcm(r_{i_1},\ldots,r_{i_k}) p^{e-1}).
                \end{cases}
                \end{equation*}
                Hence, $a_P^{\ast}(\phi,n)\geq 0$ always; and if $a_P^{\ast}(\phi,n)>0$, then
                $n$ is in one of the stated forms.
            \end{case}
        \end{proof}

        \begin{rem}
            If char $K=0$, then in Theorem \ref{thm15} we, in fact, have $a_P^{\ast}(n) \geq 1$
            \emph{if and only if} $n=m$ or $n=m\lcm(r_{i_1},\ldots,r_{i_k})$ since we know precisely the conditions for $a_P(n) > a_P(1)$.
        \end{rem}

        Note that Morton and Silverman \cite[Corollary 3.3]{Silverman6} show that for $\dim{X}=b=1$,
        if $n_1 \nmid n_2$ and $n_2\nmid n_1$, then
        $\Phi^{\ast}_{n_1}(\phi)$ and $\Phi^{\ast}_{n_2}(\phi)$ have disjoint support.  They use
        this fact to construct units in $K$ called dynatomic units similar to the construction
        of cyclotomic and elliptical units.  In the general case, the non-divisibility condition may not imply disjoint supports because there are more possible forms of $n$.  In particular, $n_1=mr_1$ and $n_2=mr_2$ could satisfy the divisibility condition, but they do not have disjoint support.

    \section{Formal $n$-periodic points of multiplicity one are primitive $n$-periodic points.} \label{formal_periodic}

        In this section we use the detailed description of the multiplicities from Section
        \ref{sect1} to show that periodic points of formal period $n$ with $n \neq 0 $ in $K$ and
        multiplicity one have primitive period $n$, generalizing \cite[Theorem 2.5]{Morton}.
        \begin{thm} \label{thm17}
            If $P$ is a primitive $m$-periodic point for $\phi$, then $a_P^{\ast}(n) \geq 2$ for all integers $n > m$
            with $\Char{K} \nmid n$ and $a_P^{\ast}(n) \neq 0$.
        \end{thm}
        \begin{proof}
            Let $n > m \geq 1$ be any integer for which $a_P^{\ast}(n) \neq 0$.  Since
            \begin{equation*}
                a_P^{\ast}(n) = \sum_{d \mid n} \mu\left(\frac{n}{d}\right)a_P(d)
            \end{equation*}
            and $a_P(d) \neq 0$ only for $m \mid d$, we must have $m$ divides $n$.
            Computing $a_P^{\ast}(\phi,n)$ in terms of $\phi^m$, we know that
            \begin{equation*}
                a_P^{\ast}(\phi,n)= a_P^{\ast}(\phi^m,n/m).
            \end{equation*}
            Hence, we may replace $\phi$ by $\phi^m$ and $n$ by $n/m$ and assume that $P$ is a fixed point.
            From Theorem \ref{thm15} we know that for $a_P^{\ast}(n) \neq 0$ and $\Char{K} \nmid
            n$ we have that $n$ is of the form
            \begin{equation*}
                n = \lcm(r_{i_1}, \ldots, r_{i_k}).
            \end{equation*}

\setcounter{case}{0}
            \begin{case} $n=r_i$ for some $1 \leq i \leq b$ (in other words, $k=1$).

                If $\lambda_i$ is in a Jordan block of $d\phi_P$ of size $>1$ then consider as $i$ the first row of the Jordan block.  Let $\beta$ be the size of the Jordan block.  In other words $x_i, \ldots, x_{i+\beta}$ are the rows of the Jordan block.  Let
                \begin{equation*}
                    \delta = \begin{cases}
                        i & \beta=1\\
                        i+\beta & \beta >1.
                    \end{cases}
                \end{equation*}

        \setcounter{subcase}{0}
                \begin{subcase} \label{case1} $\lambda_j \neq 1$ and $\lambda_j^{n} \neq 1$ for all $j \neq i$.

                    We have $a_P(1) = 1$ and need to compute $a_P(n)$.  We know
                    \begin{equation*}
                        \begin{cases}
                            x_j \in \supp(\phi_j^{(n)}(\textbf{x})-x_j) & j \neq i, \beta =1\\
                            x_{j+1} \in \supp(\phi_j^{(n)}(\textbf{x})-x_j) & i \leq j < i+\beta, \beta > 1\\
                            x_j \in \supp(\phi_j^{(n)}(\textbf{x})-x_j) & j \not\in \{i,\ldots,i+\beta\}, \beta >1
                        \end{cases}
                    \end{equation*}
                    and, using Lemma \ref{lem19} for the description of the coefficients of a monomial after iteration, we know that
                    \begin{equation} \label{eq2}
                            x_i^2 \not\in \supp(\phi_{\delta}^{(n)}(\textbf{x})-x_{\delta}).
                    \end{equation}
                    With the appropriate choice of admissible monomial ordering, we have $x_j$ for $j \neq i$ is a leading term of one of the $\phi_k^{(n)}(\textbf{x})-x_k$ for $k \neq \delta$.  Since all of these leading terms are relatively prime they are part of the generating set of a standard basis and we need only consider the monomial $x_i^e \in \supp(\phi_{\delta}^{(n)}(\textbf{x})-x_{\delta})$.  From (\ref{eq2}) we must have $e \geq 3$.
                    So then we have
                    \begin{equation*}
                        LT(I_n) \subset \{x_1,\ldots,x_{i-1},x_i^3,x_{i+1},\ldots,x_b\}.
                    \end{equation*}
                    By Lemma \ref{lem13}, we have added at least $\{x_i,x_i^2\}$ to the complement of the leading term ideal and so
                    \begin{equation*}
                        a_P^{\ast}(n) \geq 2.
                    \end{equation*}
                \end{subcase}
                \begin{subcase} $\lambda_j^n =1$ for some $j \neq i$.

                    We have $x_i$ in $LT(I_d)$ for any $d < n$ but not in $LT(I_n)$ and
                    $x_j \not\in LT(I_n)$.  So we have added at least
                    \begin{equation*}
                        \{x_i,x_ix_j\}
                    \end{equation*}
                    to the complement of $LT(I_n)$, and by Lemma \ref{lem13} we have
                    \begin{equation*}
                        a_P^{\ast}(n) \geq 2.
                    \end{equation*}
                \end{subcase}
            \end{case}
            \begin{case} $k >1$.

                We have that $n=\lcm(r_{i_1}, \ldots, r_{i_k})$.
        \setcounter{subcase}{0}
                \begin{subcase} $\lambda_j \neq1 $ for $j \not\in\{i_1,\ldots,i_k\}$.

                    We have $a_P(1) = 1$.  From Case \ref{case1} we know that
                    $a_P(r_i) \geq 3$ for each $i \in \{i_1,\ldots,i_k\}$.  Hence, we add at least
                    \begin{equation*}
                        \{x_{i_1}\cdots x_{i_k},x_{i_1}^2\cdots x_{i_k}\}
                    \end{equation*}
                    to the complement of the $LT(I_n)$.  So by Lemma \ref{lem13} we have
                    \begin{equation*}
                        a_P^{\ast}(n) \geq 2.
                    \end{equation*}
                \end{subcase}
                \begin{subcase} $\lambda_j = 1$ for some $j \not\in \{i_1,\ldots,i_k\}$.

                    We know $x_j \not\in LT(I_1)$ and hence $x_j \not\in LT(I_n)$.  Additionally,
                    $x_{i_1}\cdots x_{i_k} \in LT(I_h)$ for $h \mid n$ with $h < n$, but
                    $x_{i_1}\cdots x_{i_k} \not\in LT(I_n)$ since $r_{i_t}$ divides $n$ for each
                    $1 \leq t \leq k$.  Consequently, we add at least
                    \begin{equation*}
                        \{x_{i_1}\cdots x_{i_k},x_{i_1}\cdots x_{i_k}x_j\}
                    \end{equation*}
                    to the complement of $LT(I_n)$.  So by Lemma \ref{lem13} we have
                    \begin{equation*}
                        a_P^{\ast}(n) \geq 2.
                    \end{equation*}
                \end{subcase}
            \end{case}
        \end{proof}

        \begin{exmp}
            Theorem \ref{thm17} does not hold for $\Char{K} \mid n$.  In other words, we may have $a_P^{\ast}(n) = 1$, but $P$ is a periodic point of primitive period strictly less than $n$ if $\Char{K} \mid n$. For example, consider $\Char{K} = 3$, $\dim{X}=2$, and $\phi:X \to X$ defined near a fixed point $P$ as
            \begin{align*}
                \phi_1(x_1,x_2) &= x_1 + x_1^2 + x_1x_2 \\
                \phi_2(x_1,x_2) &= 2x_2 + x_1^2.
            \end{align*}
            Then with the monomial ordering $x_2 < x_1$, the leading term ideal is generated by $\{x_1^2,x_2\}$
            and, hence, $a_P(1) = 2$.  Iterating, we have
            \begin{align*}
                \phi^{(3)}_1(x_1,x_2) &= x_1 + x_1^3 + x_1x_2 + \text{higher order terms} \\
                \phi^{(3)}_2(x_1,x_2) &= 2x_2 + x_1^2  + \text{higher order terms}.
            \end{align*}
            Then we have the leading term ideal is generated by $\{x_1^3,x_2\}$
            and, hence, $a_P(3) = 3$.  Then computing
            \begin{equation*}
                a_P^{\ast}(3) = a_P(3) - a_P(1) = 1,
            \end{equation*}
            but $P$ is a fixed point for $\phi$.
        \end{exmp}

\section{Properties and consequences}
    Unless otherwise stated, we assume that $X$ is a non-singular, irreducible, projective variety of dimension $b$ defined over $K$ and that $\phi:X \to X$
    is a morphism defined over $K$ such that $\phi^n$ is non-degenerate.

   \subsection{Basic properties} \label{basic_properties}
       \begin{prop} \label{prop16}
            Let $m,n \geq 1$ be integers such that $\phi^{mn}$ is non-degenerate.  Then
            \begin{enumerate}
                \item \label{prop16a} If $a_P(n) >0$, then $a_P(mn)>0$ for all $m$.
                \item \label{prop16b} If $P$ is a periodic point of primitive period $n$ for $\phi$,
                    then $a^{\ast}_P(n) \neq 0$.  In particular, points of primitive
                    period $n$ are points of formal period $n$.
                \item \label{prop16c} $a_P(n) = \sum_{d \mid n} a_P^{\ast}(d)$.
                \item \label{prop16d} If $a_P(n) > 0$, then for $m$ the primitive period of $P$ for $\phi$
                    we have $a_P^{\ast}(m) >0$, for all $d < m$ we have $a_P^{\ast}(d) =0$,
                    and $a_P^{\ast}(\phi,n) = a_P^{\ast}(\phi^m,n/m)$.
            \end{enumerate}
        \end{prop}

        \begin{proof}
            \mbox{}
            \begin{enumerate}
                \item The multiplicity $a_P(n)>0$ implies that $P$ is a periodic point of period $n$, and,
                hence, $\phi^{n}(P) = P$.
                    \begin{equation*}
                        \phi^{n}(P) = P \Rightarrow \phi^{nm}(P) = \phi^{n}(\phi^n(\phi^n(\cdots
                        \phi^n(P)\cdots))) (m \text{ times})
                    \end{equation*}
                    But $\phi^n(P) = P$, so then
                    $\phi^{nm}(P)=P$. Hence, $P$ is also a periodic point of period $mn$, so it has non-zero
                    multiplicity in $\Phi_{mn}(\phi)$.
                \item
                    Since $\phi^{d}(P) \neq P$ for all $d < n$, we have
                    \begin{equation*}
                        a_P(d) = 0 \quad \text{for all } d < n.
                    \end{equation*}
                    So we have that
                    \begin{equation*}
                        a_P^{\ast}(n) = \sum_{d \mid n} \mu\left(\frac{n}{d}\right)a_P(d) = a_P(n) \neq 0,
                    \end{equation*}
                    where the last inequality comes from the fact that $P$ is a periodic point of period $n$.

                \item The definition of $\Phi^{\ast}_n(\phi)$ is
                    \begin{equation*}
                        \Phi^{\ast}_n(\phi) = \sum_{P \in X} a_P^{\ast}(n) (P).
                    \end{equation*}
                    We also have
                    \begin{equation*}
                        a_P^{\ast}(n) = \sum_{d \mid n} \mu\left(\frac{n}{d}\right)a_P(n).
                    \end{equation*}
                    We can apply M\"obius inversion to get
                    \begin{equation*}
                        a_P(n) = \sum_{d \mid n} a_P^{\ast}(d),
                    \end{equation*}
                    which gives the factorization as desired.

                \item
                    The multiplicity $a_P(n)>0$ implies that $\phi^n(P)=P$ and, hence, that $P$ is a periodic point.  Consequently, $P$ has some primitive period $m \leq n$.  By (\ref{prop16b}), $m$ satisfies $a_P^{\ast}(m) >0$.  It is the minimal such value because for any $d<m$ we have that $P$ is not a periodic point of period $d$ and, hence, $a_P(d) = 0$.  So we have $a_P^{\ast}(d) = 0$ for $d < m$.  Finally, computing $a_P^{\ast}(\phi,n)$ in terms of $\phi^m$ we have
                    \begin{align*}
                        a_P^{\ast}(\phi,n) &= \sum_{d \mid n \text{ with } m \mid d}
                        \mu\left(\frac{n}{d}\right) a_P(\phi,d) \\
                        &= \sum_{d\mid(n/m)} \mu\left(\frac{n}{md}\right)a_P(\phi,md) \\
                        &= \sum_{d \mid n/m} \mu\left(\frac{B}{d}\right)a_P(\phi^m,d)\\
                        &= a_P^{\ast}(\phi^m,n/m).
                    \end{align*}
            \end{enumerate}
        \end{proof}

        In the next proposition, we summarize some of the facts about $a_P^{\ast}(n)$ in terms
        of $\Phi^{\ast}_n(\phi)$.
        \begin{prop} \label{prop18}
            Let $m,n \geq 1$ be integers with $\phi^{mn}$ non-degenerate.
            \begin{enumerate}
                \item $a_P^{\ast}(\phi,mn) \geq a_P^{\ast}(\phi^m,n)$.
                \item If $(n,m)=1$, then $\Phi^{\ast}_n(\phi^{m})=\sum_{d \mid m}
                    \Phi^{\ast}_{nd}(\phi)$.
                \item Let $m=p^e$ for some prime $p$ and $e \geq 2$.  Then
                    $\Phi^{\ast}_{np^{e}}(\phi) = \Phi^{\ast}_{np}(\phi^{p^{e-1}})$.
                \item If $n = p_1^{e_1}\cdots p_r^{e_r}$ for distinct primes $p_1,\ldots,p_r$ with
                    $e_1,\ldots,e_r \geq 2$ and \\
                    $m=p_1^{e_1-1}\cdots p_r^{e_r-1}$, then $\Phi^{\ast}_n(\phi) =
                    \Phi^{\ast}_{p_1\cdots p_r}(\phi^m)$.
            \end{enumerate}
        \end{prop}

        \begin{proof}
            \mbox{}
            \begin{enumerate}
                \item This is clear from Lemma \ref{lem9}.
                \item We need to see that
                    \begin{equation*}
                        a_P^{\ast}(\phi^m,n) = \sum_{d \mid m}a_P^{\ast}(\phi,nd).
                    \end{equation*}
                    By the M\"{o}bius inversion formula, this is equivalent to
                    \begin{align*}
                        a_P^{\ast}(\phi,mn) = \sum_{d \mid m}
                        \mu\left(\frac{m}{d}\right)a_P^{\ast}(\phi^d,n).
                    \end{align*}
                    Computing the right-hand side, we have
                    \begin{align*}
                        \sum_{d \mid m}
                        \mu\left(\frac{m}{d}\right)a_P^{\ast}(\phi^d,n) &= \sum_{d \mid m}
                        \mu\left(\frac{m}{d}\right)\sum_{d^{\prime} \mid
                        n}\mu\left(\frac{n}{d^{\prime}}\right)a_P(\phi^d,d^{\prime})\\
                        &=\sum_{d \mid m}\sum_{d^{\prime} \mid
                        n}\mu\left(\frac{m}{d}\right)\mu\left(\frac{n}{d^{\prime}}\right)a_P(\phi,dd^{\prime})\\
                        &=\sum_{d \mid m}\sum_{d^{\prime} \mid
                        n}\mu\left(\frac{nm}{dd^{\prime}}\right)a_P(\phi,dd^{\prime})\\
                        &=\sum_{d^{\prime\prime} \mid nm}\mu\left(\frac{nm}{d^{\prime\prime}}\right)a_P(\phi,d^{\prime\prime})\\
                        &=a_P^{\ast}(\phi,nm).
                    \end{align*}

                \item This is Lemma \ref{lem9}(\ref{lem9b}).
                \item This is Lemma \ref{lem9}(\ref{lem9b}) applied to each $p_i$.
            \end{enumerate}
        \end{proof}

        \begin{prop} \label{prop17}
            $\deg(\Phi^{\ast}_n(\phi)) = \sum_{d \mid n} \mu\left(\frac{n}{d}\right) \deg(\Phi_d(\phi))$.
        \end{prop}

        \begin{proof}
            Computing:
            \begin{align*}
                \deg(\Phi^{\ast}_n(\phi))
                    &=\sum_{P \in X} \sum_{d \mid n} \mu\left(\frac{n}{d}\right) i(\Gamma_d,\Delta_X;P)\\
                    &=\sum_{d \mid n} \mu\left(\frac{n}{d}\right) \sum_{P \in X}   i(\Gamma_d,\Delta_X;P)\\
                    &= \sum_{d \mid n} \mu\left(\frac{n}{d}\right) \deg(\Phi_d(\phi)).
            \end{align*}
        \end{proof}

    \subsection{Similarities to periodic Lefschetz numbers} \label{lefschetz}
        Proposition \ref{prop17} looks remarkably similar to the definition of periodic Lefschetz
        numbers.  In this section we describe the connection.

        \begin{defn}
            Following the notation of \cite{Fagella}, define $L(\phi)$ to be the
            \emph{Lefschetz number} of $\phi$.  The \emph{periodic Lefschetz number of period $n$} is then defined as
            \begin{equation*}
                l(\phi^n) = \sum_{d \mid n} \mu\left(\frac{n}{d}\right) L(\phi^{d})
            \end{equation*}

        \end{defn}

        The Lefschetz Fixed Point Theorem states that $L(\phi^n) \neq 0$ implies that
        $\phi^n$ has a fixed point, in other words, $\phi$ has a point of period $n$, but this does not
        imply that the point is of primitive period $n$.  The periodic Lefschetz numbers
        were defined to help address this situation.  Several papers, including
        \cite{Fagella, Llibre}, have studied when $l(\phi^n)\neq 0$
        implies that there exists a periodic point of primitive period $n$.
        We will address the relationship between
        $\deg(\Phi_n)$, $\deg(\Phi^{\ast}_n)$, $L(\phi^n)$, $l(\phi^n)$, and the
        existence of period points.

        \begin{defn}
            A map $\phi$ is \emph{transversal} if $a_P(1) = 1$ for fixed points $P$.
        \end{defn}

        \begin{prop} \label{prop23}
            \mbox{}
            \begin{enumerate}
                \item \label{prop23a} $\deg{\Phi_n(\phi)} \geq L(\phi^n)$.
                \item \label{prop23b} If $\phi^n$ is transversal, then
                    \begin{enumerate}
                        \item \label{prop23b1} $a_P^{\ast}(n) = 1$ if and only if $P$ is a point of primitive
                        period $n$ for $\phi$ and $a_P^{\ast}(n)=0$ otherwise.
                        \item \label{prop23b2} $\deg(\Phi_n(\phi))$ is the number of $n$-periodic points for
                        $\phi$.
                        \item \label{prop23b3} $\deg(\Phi^{\ast}_n(\phi))$ is the number of primitive $n$-periodic
                        points for $\phi$.  In particular, if $\deg(\Phi^{\ast}_n(\phi)) \neq 0$, then there
                        exists a periodic point of primitive period $n$.
                    \end{enumerate}
            \end{enumerate}
        \end{prop}

        \begin{proof}
            \mbox{}
            \begin{enumerate}
                \item  Recall from the Lefschetz-Hopf Theorem that we may compute the Lefschetz number as
                    \begin{equation*}
                        L(\phi) = \sum_{P \in \text{Fix}(\phi)} \text{ind}(\phi,P)
                    \end{equation*}
                    where $\text{ind}(\phi,P)$ is the Poincar\'e index of $\phi$ at $P$.  So
                    $L(\phi)$ is the sum of the multiplicities of the fixed points of $\phi$ with either a negative or positive sign.

                \item
                    \begin{enumerate}
                        \item The map $\phi^n$ is transversal implies that $\phi^d$ is transversal for all $d \mid n$
                            and hence $a_P(d) = 1$ for all periodic points $P$ of period $d \mid n$.  Therefore,
                            if the primitive period of $P$ is $n$, then we have $a_P^{\ast}(n) = a_P(n) = 1$
                            since $a_P(d) = 0$ for $d < n$.

                            Assume that $P$ is a periodic point of primitive period $m \mid n$ and compute
                            \begin{equation*}
                                a_P^{\ast}(m) = a_P^{\ast}(\phi^m,1) = a_P(\phi^m,1)=1.
                            \end{equation*}
                            Since $a_P^{\ast}(m) = a_P^{\ast}(\phi^m,1)$, we may replace $\phi$ by $\phi^m$ and assume that $m=1$.  Now computing $a_P^{\ast}(n)$ we have
                            \begin{equation*}
                                a_P^{\ast}(n) = \sum_{d \mid n}\mu\left(\frac{n}{d}\right)a_P(d) =\sum_{d \mid n}\mu\left(\frac{n}{d}\right)1 =0
                            \end{equation*}
                            by properties of the M\"obius function.

                        \item Properties (\ref{prop23b2}) and (\ref{prop23b3}) follow directly from the definition of transversal and (\ref{prop23b1}).
                    \end{enumerate}
            \end{enumerate}
        \end{proof}
        \begin{rem}
            Proposition \ref{prop23}(\ref{prop23b}) is similar to \cite[Theorem A]{Fagella}.
        \end{rem}

    \subsection{Applications} \label{applications}
        \begin{prop} \label{prop19}
            There are only finitely many points of primitive period $n$ for any fixed $n$
            with $\phi^n$ non-degenerate.
        \end{prop}

        \begin{proof}
            Fix any integer $n \geq 1$ with $\phi^n$ non-degenerate.  Proposition \ref{prop17} provides a formula for the
            degree of $\Phi^{\ast}_n(\phi)$.  Since $\phi^n$ is assumed to be non-degenerate,
            B\'{e}zout's Theorem states that $\Gamma_d$ and $\Delta$ intersect in a finite number
            of points for all $d \mid n$; in other words, $\deg(\Phi_d(\phi))$ is finite.  Hence,
            $\deg(\Phi^{\ast}_n(\phi))$ is finite, so there can only be finitely many
            primitive $n$-periodic points.
        \end{proof}

        \begin{thm} \label{thm21}
            There exists $M > 0$ such that for all $q$ prime and $q >M$,
            $\deg(\Phi^{\ast}_q(\phi)) \neq 0$ implies that there exists a periodic point with
            primitive period $q$ for $\phi$.
        \end{thm}

        \begin{proof}
            We want to show that there exists a $P$ with $a_P^{\ast}(q) \neq 0$ that is
            a primitive $q$-periodic point.  We know that for $q$ prime we have
            \begin{equation*}
                \deg(\Phi^{\ast}_q(\phi)) = \deg(\Phi_q(\phi)) - \deg(\Phi_1(\phi)).
            \end{equation*}
            There are only finitely many fixed points for $\phi$ by Proposition \ref{prop19},
            and for each fixed point only finitely many $n$ relatively prime to the characteristic of $K$ such that $a_P(n) > a_P(1)$ by Theorem \ref{thm15}.  Hence, after excluding those finitely many numbers (including the characteristic of $K$), each time $\deg(\Phi_q(\phi)) > \deg(\Phi_1(\phi))$
            the additional degree comes from at least one periodic point of primitive
            period $q$.
        \end{proof}

        \begin{cor} \label{cor11}
            If there are infinitely many $n \in \mathbb{Z}^{+}$ such that
            $\deg(\Phi^{\ast}_n(\phi)) \neq 0$ for $n \neq 0$ in $K$ and $\phi^n$ is
            non-degenerate, then there exists $P \in X$ with an arbitrarily large primitive
            period for $\phi$, and $\phi$ has infinitely many periodic points.
        \end{cor}

        \begin{proof}
            By assumption, we have infinitely many primes $q$ with $\deg(\Phi^{\ast}_q(\phi))\neq
            0$.  Applying Theorem \ref{thm21}, we then have infinitely many primes $q$
            with a periodic point of primitive period $q$.
        \end{proof}

        \begin{rem}
            Corollary \ref{cor11} appears to be similar to applications of periodic Lefschetz numbers such as those in \cite{Dold, Fagella}.
        \end{rem}

        \begin{thm} \label{thm22}
            If $P$ is a fixed point of $\phi$, then the sequence
            \begin{equation*}
                \mathop{\{a_P(n)\}_{n \in \mathbb{N}}}_{\hspace{52pt} \Char{K} \nmid n}
            \end{equation*}
            is bounded.
        \end{thm}

        \begin{proof}
            From Theorem \ref{thm15} we have that for a fixed point $P$ for $\phi$, $a_P(n)
            \neq a_P(1)$ for only finitely many $n$ with $\Char{K}\nmid n$.
            Hence the sequence must be bounded.
        \end{proof}

        \begin{cor} \label{cor10}
            If $\deg(\Phi_n(\phi))$ is unbounded for $\Char{K} \nmid n$, then there are infinitely
            many periodic points for $\phi$ and, hence, periodic points with arbitrarily large primitive periods.
        \end{cor}

        \begin{proof}
            Consider the prime numbers $q \in \mathbb{Z}$ with $q \neq \Char{K}$.  We know that
            $\deg(\Phi_q(\phi))$ is
            unbounded, and the only contributions come from fixed points or points of primitive
            period $q$.  Since the sequence $a_P(q)$ is bounded for all fixed points $P$, there
            must be contributions to $\deg(\Phi_q(\phi))$ from periodic points of primitive
            period $q$ for infinitely many primes $q$.
        \end{proof}

        \begin{rem}
            Theorem \ref{thm22} and Corollary \ref{cor10} are similar to \cite{Shub}.
        \end{rem}

    \subsection{Wehler K3 surfaces} \label{wehler}
        A Wehler K3 surface $S \subset \mathbb{P}^{2} \times \mathbb{P}^{2}$ is
        a smooth surface given by the intersection of an effective divisor of degree (1,1) and an effective divisor of degree (2,2).  Wehler \cite[Theorem 2.9]{Wehler} shows that
        these surfaces have an infinite automorphism group, from which we have
        dynamical systems. These dynamical systems were studied in \cite{Silverman4, Silverman5}.

        \begin{thm} \label{thm19}
            Dynamical systems on Wehlers K3 surfaces have points with arbitrarily large primitive period and
            infinitely many periodic points.  In particular, there exists a constant $M$
            such that for all primes $q > M$ there exists a periodic point of primitive
            period $q$.
        \end{thm}

        \begin{proof}
            From \cite[page 358]{Silverman5} we know that the Lefschetz numbers of the maps
            $\phi^k = (\sigma_1 \circ \sigma_2)^k$ are given by
            \begin{equation*}
                L(\phi^k) = (2+\sqrt{3})^{2k}  + (2+\sqrt{3})^{-2k} + 22.
            \end{equation*}
            So we have
            \begin{equation} \label{eq58}
                L(\phi^k) \geq 2^{2k}.
            \end{equation}
            By Proposition \ref{prop23}(\ref{prop23a})
            \begin{equation}\label{eq57}
                \deg(\Phi_k(\phi)) \geq L(\phi^k),
            \end{equation}
            hence, we have that $\deg(\Phi_k(\phi))$ is unbounded as $k$ increases.  Applying
            Corollary \ref{cor10}, we have the result.

            To show the second portion, recall that
            \begin{equation*}
                \deg(\Phi^{\ast}_q(\phi)) = \deg(\Phi_q(\phi)) - \deg(\Phi_1(\phi)).
            \end{equation*}
            In other words, whenever $\deg(\Phi_q(\phi)) > \deg(\Phi_1(\phi))$
            we have $\deg(\Phi^{\ast}_q(\phi)) \neq 0$.  Combining (\ref{eq58})
            and (\ref{eq57}), we have that for $k$ larger than some constant $C$ we have
            $\deg(\Phi^{\ast}_q(\phi)) \neq 0$.  Applying Theorem \ref{thm21} now gives the
            desired result.
        \end{proof}

        \begin{defn}
            Let $S$ be a Wehler K3 surface and let $\mathcal{A}$ be the subgroup of the automorphism group of $S$
            generated by $\sigma_1$ and $\sigma_2$.  Let $B_k \subset \mathcal{A}$ be the cyclic subgroup generated by
            $\phi^k = (\sigma_1\circ \sigma_2)^k$.
            Let $\mathcal{A}_P = \{\phi \in \mathcal{A} \mid \phi(P) =P\}$.
            Let $S[B] = \{P \in S(K) \mid \mathcal{A}_P = B\}$. Recall that we are
            assuming $K$ is algebraically closed.
        \end{defn}
        The following proposition addresses a remark of Silverman from \cite[page 358]{Silverman5}.

        \begin{prop}
            $\#S[B_q] \to \infty$ as $q \to \infty$ for $q$ prime.
        \end{prop}

        \begin{proof}
            From Theorem \ref{thm19} we have that there are periodic points of infinitely large
            prime primitive period and, in particular, periodic points of prime primitive period for all primes larger than some constant $M$.
            Hence, $S[B_q]$ will increase as $q$ increases.
        \end{proof}

    \subsection{Morphisms of projective space} \label{pn}
        We also apply our results to morphisms of projective space.  Let
        $\phi:\mathbb{P}^N \to \mathbb{P}^{N}$ be a morphism of degree $d$.  We need to compute
        the intersection number for $\Delta$ and $\Gamma_{\phi^n}$, which are contained in
        $\mathbb{P}^{N} \times \mathbb{P}^{N}$.

        Let $D_1$ and $D_2$ be the pullbacks in $\mathbb{P}^{N} \times \mathbb{P}^{N}$
        of a hyperplane class $D$ in $\mathbb{P}^{N}$ by the first and second projections, respectively.
        \begin{prop} \label{prop20}
            Let $\Delta$ and $\Gamma_{\phi^n}$ be defined as above.
            \begin{enumerate}
                \item \label{prop20a} The class of $\Delta$ is given by
                    \begin{equation*}
                        \sum_{j=0}^{N} D_1^{N-j}D_2^{j}.
                    \end{equation*}

                \item \label{prop20b} The class of $\Gamma_{\phi^n}$ is given by
                    \begin{equation*}
                        \sum_{j=0}^{N} d^{N-j}D_1^{N-j}D_2^{j}.
                    \end{equation*}
            \end{enumerate}
        \end{prop}

        \begin{proof}
            \mbox{}
            \begin{enumerate}
                \item By the Kunneth formula, the diagonal must be a class in
                    \begin{equation*}
                        H_N(\mathbb{P}^{N} \times \mathbb{P}^{N}) =
                        \sum_{j=0}^N H_{N-j}(\mathbb{P}^{N}) \otimes H_j(\mathbb{P}^{N}).
                    \end{equation*}

                    Now, $H_{N-j}(\mathbb{P}^{N}) \otimes H_j(\mathbb{P}^{N})$ is a 1-dimensional space
                    for all $0 \leq j \leq N$, spanned by the Poincar\'{e} dual of
                    $D_1^{N-j}D_2^j$.  We can write
                    \begin{equation*}
                        \Delta = \sum_{j=0}^{N} a_j D_1^{N-j} D_2^j.
                    \end{equation*}
                    To determine the coefficient $a_j$, we should intersect $\Delta$ with the dual of $D_1^{N-j}D_2^j$.
                    This is $D_1^jD_2^{N-j}$.
                    So let $i_{\Delta}:\Delta \hookrightarrow \mathbb{P}^{N} \times \mathbb{P}^{N}$ and compute
                    \begin{equation*}
                        (D_1^{N-j}D_2^j)\cdot(\Delta) = i_{\Delta}^{\ast}(D_1^{N-j}D_2^j)\cdot \mathbb{P}^N = D^N \cdot \mathbb{P}^{N} = 1
                    \end{equation*}
                    using the fact that $i_{\Delta}^{\ast}(D_1) = i_{\Delta}^{\ast}(D_2) = D$, a hyperplane class on $\mathbb{P}^N$.

                \item  Again, by the Kunneth formula, the graph must be a class in
                    \begin{equation*}
                        H_N(\mathbb{P}^{N} \times \mathbb{P}^{N}) =
                        \sum_{j=0}^N H_{N-j}(\mathbb{P}^{N}) \otimes H_j(\mathbb{P}^{N}),
                    \end{equation*}
                    and we can write
                    \begin{equation*}
                        \Gamma_{\phi^n} = \sum_{j=0}^{N} a_j D_1^{N-j}D_2^j.
                    \end{equation*}
                    Let $i_{\Gamma_n}:\Gamma_n \hookrightarrow \mathbb{P}^{N} \times \mathbb{P}^{N}$.
                    To determine the coefficients $a_j$, we compute
                    \begin{align*}
                        (D_1^{N-j}D_2^j)\cdot(\Gamma_{\phi^n}) &= i_{\Gamma_n}^{\ast}(D_1^{N-j}D_2^j)\cdot \mathbb{P}^N = d^{N-j}D^N \cdot \mathbb{P}^{N} = d^{N-j}
                    \end{align*}
                    using the facts that $i_{\Gamma_n}^{\ast}(D_1) = dD$, since $\phi$ is degree $d$, and $i_{\Gamma_n}^{\ast}(D_2) = D$.
            \end{enumerate}
        \end{proof}

        \begin{prop} \label{prop21}
            A morphism $\phi: \mathbb{P}^N \to  \mathbb{P}^N$ of degree $d$ has at most
            \begin{equation*}
                d^N + d^{N-1} + \cdots + d + 1
            \end{equation*}
            fixed points.
        \end{prop}

        \begin{proof}
            By Proposition \ref{prop20}, we compute the intersection number of
            $\Gamma_{\phi}$ and $\Delta$.
            \begin{align*}
                (\Gamma_{\phi}) \cdot(\Delta) &= \left(\sum_{j=0}^{N} D_1^{N-j}D_2^{j}\right)\cdot
                    \left(\sum_{k=0}^{N} d^{N-k}D_1^{N-k}D_2^{k}\right) \\
                    &= 0 + \sum_{j=0}^{N}d^{N-j}D_1^{N}D_2^N \\
                    &= \sum_{j=0}^{N} d^j.
            \end{align*}
            Since each fixed point has multiplicity at least 1, $\sum_{j=0}^{N} d^j$ is the maximum possible number of fixed points.
        \end{proof}

        \begin{prop} \label{prop22}
            A morphism $\phi: \mathbb{P}^N \to  \mathbb{P}^N$ of degree $d$ has
            \begin{equation*}
                \deg(\Phi_n(\phi)) = \sum_{j=0}^N (d^n)^j.
            \end{equation*}
        \end{prop}

        \begin{proof}
            $\phi^n$ has degree $d^n$ so we apply Proposition \ref{prop21} to $\phi^n$.
        \end{proof}

        \begin{thm} \label{thm20}
            A morphism $\phi: \mathbb{P}^N \to  \mathbb{P}^N$ of degree $d>1$ has periodic
            points with arbitrarily large primitive periods and infinitely many periodic
            points.  In particular, there exists a constant $M$ such that for all primes
            $q>M$ there exist periodic points of primitive period $q$.
        \end{thm}

        \begin{proof}
            The degree $\deg(\Phi_n(\phi))$ is clearly unbounded from Proposition \ref{prop22}, so we
            apply Corollary \ref{cor10} to conclude the first result.

            To see the second result, notice that
            \begin{align*}
                \deg(\Phi^{\ast}_q(\phi)) &= \deg(\Phi_q(\phi)) - \deg(\Phi_1(\phi))\\
                    &=((d^q)^N + \cdots + (d^q) + 1) - (d^N + \cdots + d + 1) \\
                    &= ((d^q)^N - d^N) + \cdots (d^q - d)\\
                    &> 0.
            \end{align*}
            Hence, we apply Theorem \ref{thm21} to conclude the
            result.
        \end{proof}

\setlength{\bibsep}{0pt}
\bibliography{masterlist}
\bibliographystyle{plain}

\begin{center}
Department of Mathematics and Computer Science, Amherst College, Amherst MA, 01002 USA\\
e-mail: bhutz@amherst.edu
\end{center}
\end{document}